[1994/12/01]
\documentclass[draft]{article}
\title{Extremal Hypergraphs for Ryser's Conjecture:\\
Home-Base Hypergraphs}
\author{Penny Haxell\thanks{Department of Combinatorics and Optimization, University of Waterloo, Waterloo, ON, Canada. Partially supported by NSERC and by a Friedrich Wilhelm Bessel
 Award of the Alexander von Humboldt Foundation.} \quad Lothar Narins\thanks{Freie Universit\"at Berlin, Berlin, Germany. Supported by the Research Training Group 
 \emph{Methods for Discrete Structures} and the Berlin Mathematical School.} \quad 
Tibor Szab\'{o}\thanks{Freie Universit\"at Berlin, Berlin, Germany. Research partially supported by DFG within the Research Training Group \emph{Methods for Discrete Structures.}}}
\date{\today}

\usepackage{amsmath, amsthm, amssymb, tikz}
\usepackage[margin=20pt,small,bf]{caption}
\usetikzlibrary{arrows, positioning, shapes}

\newtheorem{thm}{Theorem}[section]
\newtheorem{cor}[thm]{Corollary}
\newtheorem{lem}[thm]{Lemma}
\newtheorem{prop}[thm]{Proposition}
\newtheorem{obs}[thm]{Observation}

\newtheorem{conj}{Conjecture}
\newtheorem*{claim}{Claim}

\theoremstyle{definition}
\newtheorem{defn}[thm]{Definition}
\newtheorem*{assm}{Assumptions}
\newtheorem*{ih}{Induction Hypothesis}

\theoremstyle{remark}

\newtheorem*{notation}{Notation}

\numberwithin{equation}{section}

\DeclareMathOperator{\conn}{conn}
\DeclareMathOperator{\lk}{lk}

\newcommand{\N}{\mathbb{N}}
\newcommand{\R}{\mathbb{R}}

\newcommand{\cF}{\mathcal{F}}
\newcommand{\cG}{\mathcal{G}}
\newcommand{\cH}{\mathcal{H}}

\newcommand{\cM}{\mathcal{M}}

\newcommand{\cR}{\mathcal{R}}

\newcommand{\cU}{\mathcal{U}}

\newcommand{\lset}[1]{\left\{#1\right\}}
\newcommand{\set}[2]{\left\{#1 : #2 \right\}}
\newcommand{\abs}[1]{\left|#1\right|}

\newcommand{\link}[2]{\lk_{#1}(#2)}

\begin{document}
\maketitle
\markboth{Extremal Hypergraphs for Ryser's Conjecture}
{Home-Base Hypergraphs }
\renewcommand{\sectionmark}[1]{}

\begin{abstract}
Ryser's Conjecture states that any $r$-partite $r$-uniform hypergraph has a vertex cover of size at most $r - 1$ times the size of the largest matching. For $r = 2$, the conjecture is simply K\"onig's Theorem and every bipartite graph is a witness for its tightness. The conjecture has also been proven for $r = 3$ by Aharoni using topological methods, but the proof does not give information on the extremal $3$-uniform hypergraphs. Our goal in this paper is to characterize those hypergraphs which are tight for Aharoni's Theorem. 

Our proof of this characterization is also based on topological machinery, particularly utilizing results on the (topological) connectedness of the independence complex of the line graph of the link graphs of $3$-uniform Ryser-extremal hypergraphs, developed in a separate paper~\cite{HNS}. The current paper contains the second, structural hypergraph-theoretic part of the argument, where we use the information on the line graph of the link graphs to nail down the elements of a structure we call  \emph{home-base hypergraph}. While there is a single minimal home-base hypergraph with matching number $k$ for every positive integer $k \in \N$, home-base hypergraphs with matching number $k$ are far from being unique. There are infinitely many of them and each of them is composed of $k$ copies of two different kinds of basic structures, whose hyperedges can intersect in various restricted, but intricate ways.

Our characterization also proves an old and wide open strengthening of 
Ryser's Conjecture, due to  Lov\'asz, for the $3$-uniform extremal case, that is,
for hypergraphs with $\tau =2 \nu$.
\end{abstract}

\section{Ryser's Conjecture} \label{sec:rysersconjecture}

A hypergraph $\cH$ is a pair $(V, E)$, where $V = V(\cH)$ is the set of \emph{vertices}, and $E = E(\cH)$ is a {\bf multiset} of subsets of vertices called the \emph{edges} of $\cH$. The number of times a subset $e\subseteq V$ appears in $E$ is called the \emph{multiplicity} of $e$. A hypergraph is called \emph{simple} if the multiplicity of each subset is at most $1$. An edge $e \in E$ is called \emph{parallel} to an edge $f \in E$ if their underlying vertex subsets are the same. In particular, every edge is parallel to itself. If the cardinality of every edge is $r$, we call $\cH$ an \emph{$r$-graph}. A $2$-graph is called a \emph{graph}. 
 
Let $\cH$ be a hypergraph. A \emph{matching} in $\cH$ is a set of disjoint edges of $\cH$, and the \emph{matching number}, $\nu(\cH)$, is the size of the largest matching in $\cH$. If $\nu(\cH) = 1$, then $\cH$ is called \emph{intersecting}. A \emph{vertex cover} of $\cH$ is a set of vertices which intersects every edge of $\cH$. The size of the smallest vertex cover is called the \emph{vertex cover number} of $\cH$ and is denoted by $\tau(\cH)$. It is immediate to see that if $\cH$ is $r$-uniform, then the following always holds:
\[
	\nu(\cH) \leq \tau(\cH) \leq r\nu(\cH).
\]

Both inequalities are easily seen to be tight for general hypergraphs. Ryser's Conjecture (see e.g. \cite{tuza2}), originating from the early 1970's, states that the upper bound can be lowered by considering only $r$-partite hypergraphs. An $r$-graph is called \emph{$r$-partite} if its vertices can be partitioned into $r$ parts, called \emph{vertex classes}, such that every edge intersects each vertex class exactly once.

\begin{conj}[Ryser's Conjecture]
If $\cH$ is an $r$-partite $r$-graph, then
\[
	\tau(\cH) \leq (r - 1)\nu(\cH).
\]
\end{conj}

 Around the same time a much stronger conjecture was made by Lov\'asz~\cite{lovasz}.
 The conjecture states that not only do we have a vertex cover of size $(r-1)\nu(\cH)$,
 but we can obtain it by repeatedly reducing the matching number by one with the 
 removal of $r-1$ vertices.

 \begin{conj}[Lov\'asz Conjecture]
In every $r$-partite $r$-graph there exist $r-1$ vertices whose deletion reduces the 
matching number.
 \end{conj} 
 
 For $r = 2$ both conjectures are implied by K\"onig's theorem.
 For $r = 3$ Aharoni~\cite{aharoni} proved Ryser's Conjecture using topological methods. 
 The Lov\'asz Conjecture is wide open for $r\geq 3$, as is Ryser's Conjecture 
 for $r \geq 4$. For more on the history, see~\cite{HNS}.

K\"onig's theorem implies that for $r = 2$, the conjecture is tight for \emph{every} bipartite graph. For $r = 3$, Aharoni's proof does not give information on the extremal $3$-graphs. Our aim is to give a complete characterization of them. We prove the following theorem:

\begin{thm} \label{thm:characterization}
Let $\cH$ be a $3$-partite $3$-graph. Then $\tau(\cH) = 2\nu(\cH)$ if and only if $\cH$ is a home-base hypergraph.
\end{thm}

Home-base hypergraphs have a restricted structure, but are far from being unique: for any given positive integer $k \in \mathbb{N}$ there are infinitely many home-base hypergraphs with matching number $k$. The precise description is given in the following subsection.

One could speculate that among $3$-partite $3$-graphs with given matching number $k$, the hardest
instances to prove the Lov\'asz Conjecture for will be the ones with vertex cover number
as large as possible, i.e. $2k$. 
Indeed, for these graphs the conjecture will have to be tight in every step.
As it turns out the Lov\'asz Conjecture for $3$-graphs  with $\tau(\cH) = 2\nu(\cH)$ 
is a relatively simple consequence of our characterization in 
Theorem~\ref{thm:characterization}.

\begin{cor} \label{cor:lovasz}
Let $\cH$ be a $3$-partite $3$-graph with $\tau(\cH) = 2\nu(\cH)$. 
Then there exists $\nu(\cH)$ pairwise disjoint pairs of vertices 
such that the removal of the union of any $k$, $1\leq k \leq \nu (\cH)$, of these pairs 
decreases the matching number by $k$.\\
In  particular there exist two vertices, the removal of which 
reduces the matching number of $\cH$.
\end{cor}

\subsection{Home-Base Hypergraphs}

To motivate our definition of home-base hypergraphs, let us start with some examples of $3$-graphs $\cH$ with $\tau(\cH) = 2 = 2\nu(\cH)$. A general example of an $r$-graph, which is tight for Ryser's Conjecture is the \emph{truncated projective plane} $F^{(r)}$. Its vertex set is constructed by taking the projective plane over the $(r - 1)$-element field and removing one point from it. The lines of the plane which were incident to this point become the vertex classes of the $r$-graph, and the rest of the lines become the edges. Since any two lines of the projective plane intersect, we have $\nu(F^{(r)}) = 1$. It is also not difficult to see that the smallest vertex covers are the vertex classes and hence $\tau(F^{(r)}) = r - 1$. Truncated projective planes exist whenever $r$ is one greater than a prime power. Luckily, $3$ is such a number, and thus we have the truncated Fano plane. Concretely, the \emph{truncted Fano-plane} is the $3$-graph $F^{(3)} = F$ with vertex set $\lset{a, b, c, x, y, z}$ and edges $abc$, $ayz$, $xbz$, and $xyc$ (here the vertex classes are $\lset{a, x}$, $\lset{b, y}$, and $\lset{c, z}$); see Figure~\ref{fig:fano}. (In all our pictures of $3$-partite
hypergraphs the vertex classes $V_1$, $V_2$, $V_3$ are drawn vertically.)

Adding parallel edges to any hypergraph does not affect the vertex cover number or the matching number. We call any $3$-graph a \emph{truncated multi-Fano plane}, if it is obtained from the truncated Fano-plane by adding an arbitrary number of parallel edges.

\begin{center}
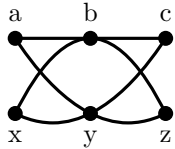
 
\begin{tikzpicture} 
[edge/.style={very thick},
vertex/.style={circle, fill, inner sep=0, minimum size=2mm}]

\draw[edge] (0,1) -- (2,1);
\draw[edge] (0,1) parabola bend (3/2,-1/8) (2,0);
\draw[edge] (0,0) parabola bend (1,1) (2,0);
\draw[edge] (0,0) parabola bend (1/2,-1/8) (2,1);
\node[vertex] at (0,1) [label=above:a] {};
\node[vertex] at (1,1) [label=above:b] {};
\node[vertex] at (2,1) [label=above:c] {};
\node[vertex] at (0,0) [label=below:x] {};
\node[vertex] at (1,0) [label=below:y] {};
\node[vertex] at (2,0) [label=below:z] {};
\end{tikzpicture}
\captionof{figure}{The truncated Fano plane.} \label{fig:fano}
\end{center}

However, the truncated Fano-plane is not minimal, since removing any edge from it yields another example of an intersecting hypergraph which cannot be covered by a single vertex. To be concrete, let $C$ be the hypergraph on the vertex set $\lset{a, b, c, x, y, z}$ and edges $ayz$, $xbz$, and $xyc$. (This hypergraph is called the \emph{loose 3-cycle}.) Note that three of the vertices have degree $2$ and three have degree $1$. One can extend $C$ by adding edges (perhaps containing new vertices) which contain two of the degree $2$ vertices and still obtain an intersecting hypergraph (and obviously the vertex cover number does not decrease). This creates a family of edges which is intersecting simply because they all contain two of the vertices $x$, $y$, and $z$. Thus this family is determined by the set $R = \lset{x, y, z}$.

\begin{center}
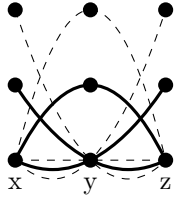

\begin{tikzpicture}
[edge/.style={very thick},
halfedge/.style={dashed},
vertex/.style={circle, fill, inner sep=0, minimum size=2mm}]

\draw[edge] (0,1) parabola bend (3/2,-1/8) (2,0);
\draw[edge] (0,0) parabola bend (1,1) (2,0);
\draw[edge] (0,0) parabola bend (1/2,-1/8) (2,1);
\draw[halfedge] (0,2) parabola bend (3/2,-1/4) (2,0);
\draw[halfedge] (0,0) parabola bend (1,2) (2,0);
\draw[halfedge] (0,0) parabola bend (1/2,-1/4) (2,2);
\draw[halfedge] (0,0) -- (2,0);
\foreach \x in {0, 1, 2}
	\foreach \y in {1, 2}
		\node[vertex] at (\x,\y) {};
\node[vertex] at (0,0) [label=below:x] {};
\node[vertex] at (1,0) [label=below:y] {};
\node[vertex] at (2,0) [label=below:z] {};
\end{tikzpicture}
\captionof{figure}{The truncated Fano plane minus one edge, with possible additional edges drawn in dashed lines.}
\end{center}

We say that a $3$-partite $3$-graph $\cH$ is \emph{Ryser-extremal}, if $\tau(\cH) = 2 \nu(\cH)$. Our hope would be that every Ryser-extremal $3$-graph is made up of such $R$-families and truncated multi-Fano-planes. This is indeed the case, but the edges of these substructures can intersect in various ways. 

The simplest way to describe a home-base hypergraph $\cH$ with $\nu(\cH)=k$ is as follows. Start with a set of $k$ disjoint hypergraphs, each of which is a copy of $F$ or a loose 3-cycle $C$ as above. Now add any number of additional edges $e$, each of which intersects some 
$C$ in at least two of its degree-2 vertices. (The third vertex of $e$ is arbitrary in the remaining vertex class of ${\cal H}$.) This does capture the notion of home-base
hypergraph, however for our inductive proof the following series of 
more technical definitions will be needed.

\begin{defn} \label{def:frpartition}
Let $\cH$ be a $3$-partite $3$-graph. An \emph{FR-partition} of $\cH$ is a triple $(\cF, \cR, W)$ with $\cF, \cR \subseteq 2^{V(\cH)}$ and $W \subseteq V(\cH)$ which satisfies the following conditions:
\begin{enumerate}
	\renewcommand{\theenumi}{(\arabic{enumi})}
	\renewcommand{\labelenumi}{\theenumi}
	\item \label{fr:partition} $\cF \cup \cR \cup \lset{W}$ is a partition of the vertices of $\cH$,
	\item \label{fr:f} For each $F \in \cF$, the induced hypergraph $\cH|_F$ is isomorphic to a truncated multi-Fano plane,
	\item \label{fr:r} Each $R \in \cR$ is a three-vertex set with one vertex from each vertex class of $\cH$,
	\item \label{fr:size} $\abs{\cF \cup \cR} = \nu(\cH)$.
\end{enumerate}
\end{defn}

Note that $\cF$ is a $6$-graph and $\cR$ is a $3$-graph.

\begin{defn} \label{def:matchable}
$\cH$ be a $3$-partite $3$-graph with vertex classes $V_1$, $V_2$, and $V_3$, and let $(\cF, \cR, W)$ be an FR-partition of $\cH$. For each vertex class $V_i$, we define a bipartite graph $B_i$ with vertex classes $\cR$ and $W \cap V_i$ and with an edge between $R \in \cR$ and $w \in W \cap V_i$ precisely when there is an edge of $\cH$ containing $w$ and two vertices of $R$. The partition $(\cF, \cR, W)$ is called \emph{matchable} if each $B_i$ has a matching saturating $\cR$.
\end{defn}

An example of a non-matchable FR-partition is given in the following picture, where the boxes correspond to two $R$'s and the unboxed vertices are in $W$:

\begin{center}
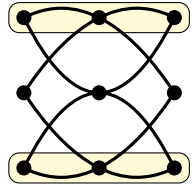

\begin{tikzpicture}
[edge/.style={very thick},
vertex/.style={circle, fill, inner sep=0, minimum size=2mm},
r/.style={rounded corners, draw, fill=yellow!20, minimum height=4mm, minimum width=24mm, inner sep=0mm}]

\node[r] at (1,0) {};
\node[r] at (1,2) {};
\draw[edge] (0,1) parabola bend (3/2,-1/8) (2,0);
\draw[edge] (0,0) parabola bend (1,1) (2,0);
\draw[edge] (0,0) parabola bend (1/2,-1/8) (2,1);
\draw[edge] (0,1) parabola bend (3/2,17/8) (2,2);
\draw[edge] (0,2) parabola bend (1,1) (2,2);
\draw[edge] (0,2) parabola bend (1/2,17/8) (2,1);
\foreach \x in {0, 1, 2}
	\foreach \y in {0, 1, 2}
		\node[vertex] at (\x,\y) {};
\end{tikzpicture}
\captionof{figure}{An unmatchable FR-partition.}
\end{center}

\begin{defn} \label{def:edgehome}
An FR-partition $(\cF, \cR, W)$ of $\cH$ is said to have the \emph{edge-home} property if every edge of $\cH$ is either in $\cH|_F$ for some $F \in \cF$ or contains two vertices from some $R \in \cR$.
\end{defn}

\begin{defn} \label{def:homebase}
A matchable FR-partition with the edge-home property is called a \emph{home-base partition}. $\cH$ is called a \emph{home-base hypergraph} if it has a home-base partition.
\end{defn}

\begin{notation}
For each $F \in \cF$, we call an edge an $F$-edge if it is in $\cH|_F$. For each $R \in \cR$, we call an edge an $R$-edge if it contains two vertices from $R$. We call an edge an $\cF$-edge if it is an $F$-edge for some $F \in \cF$, and call an edge an $\cR$-edge if it is an $R$-edge for some $R \in \cR$.
\end{notation}

Here follows an example of a home-base hypergraph. The boxes correspond to members of $\cF$ or $\cR$, and the unboxed vertices are in $W$. The bolded edges are the edges of $\cH|_F$ for some $F \in \cF$ or the edges corresponding to the edges of arbitrarily chosen matchings saturating $\cR$ in the auxiliary bipartite graphs $B_i$.

\begin{center}
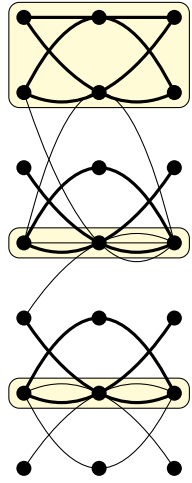

\begin{tikzpicture}
[edge/.style={very thick},
thinedge/.style={},
vertex/.style={circle, fill, inner sep=0, minimum size=2mm},
f/.style={rounded corners, draw, fill=yellow!20, minimum height=14mm, minimum width=24mm, inner sep=0mm},
r/.style={rounded corners, draw, fill=yellow!20, minimum height=4mm, minimum width=24mm, inner sep=0mm}]

\node[f] at (1,11/2) {};
\node[r] at (1,3) {};
\node[r] at (1,1) {};
\draw[edge] (0,6) -- (2,6);
\draw[edge] (0,6) parabola bend (3/2,39/8) (2,5);
\draw[edge] (0,5) parabola bend (1,6) (2,5);
\draw[edge] (0,5) parabola bend (1/2,39/8) (2,6);
\draw[edge] (0,4) parabola bend (3/2,23/8) (2,3);
\draw[edge] (0,3) parabola bend (1,4) (2,3);
\draw[edge] (0,3) parabola bend (1/2,23/8) (2,4);
\draw[edge] (0,2) parabola bend (3/2,7/8) (2,1);
\draw[edge] (0,1) parabola bend (1,2) (2,1);
\draw[edge] (0,1) parabola bend (1/2,7/8) (2,2);
\draw[thinedge] (0,0) parabola bend (3/2,9/8) (2,1);
\draw[thinedge] (0,1) parabola bend (1,0) (2,1);
\draw[thinedge] (0,1) parabola bend (1/2,9/8) (2,0);
\draw[thinedge] (0,2) parabola bend (3/2,25/8) (2,3);
\draw[thinedge] (0,3) -- (2,3);
\draw[thinedge] (0,3) parabola bend (1,5) (2,3);
\draw[thinedge] (0,5) parabola bend (3/2,11/4) (2,3);
\foreach \x in {0, 1, 2}
	\foreach \y in {0, ..., 6}
		\node[vertex] at (\x,\y) {};
\end{tikzpicture}
\captionof{figure}{A home-base hypergraph with its home-base partitition.}
\end{center}

We can easily see one direction of Theorem~\ref{thm:characterization}:

\begin{prop} \label{prop:homebasetight}
If $\cH$ has a home-base partition $(\cF, \cR, W)$, then $\tau(\cH) = 2\nu(\cH)$.
\end{prop}

\begin{proof}
Let $T \subseteq V(\cH)$ be a vertex cover. We aim to show that it has size at least $2\nu(\cH) = 2\abs{\cF \cup \cR}$. Since the partition is matchable, each of the auxiliary bipartite graphs $B_1$, $B_2$, and $B_3$ have matchings saturating $\cR$, say $M_1$, $M_2$, and $M_3$, respectively. Then each $R = \lset{r_1, r_2, r_3} \in \cR$ has three $W$-vertices, $w^R_i \in V_i$ assigned to it, so that $Rw^R_i \in M_i$, which means that $w^R_i r_j r_k$ are edges for each choice of $\lset{i, j, k} = \lset{1, 2, 3}$. So consider only the edges of this form together with the edges of $\cH|_F$ for each $F \in \cF$. Each set of edges for each $R \in \cR$ and $F \in \cF$ is disjoint from the other sets, so any vertex cover must cover each set with different vertices. Since each such set forms an intersecting $3$-partite $3$-graph with vertex cover number $2$, $T$ must have at least two vertices for each $R \in \cR$ and each $F \in \cF$, giving a total of at least $2\abs{\cR \cup \cF} = 2\nu(\cH)$ vertices as required. This shows $\tau(\cH) \geq 2\nu(\cH)$. Since Ryser's Conjecture is true for $3$-partite $3$-graphs, we have $\tau(\cH) = 2\nu(\cH)$.
\end{proof}

Note that we did not make use of the edge-home property in this proof. This property is necessary however to ensure that if a home-base partition exists, then it is unique. 
Uniqueness is not necessary for our proof of the main theorem, for its proof 
we refer to~\cite{lotharsthesis}.

The definition of home-base hypergraphs together with  
Theorem~\ref{thm:characterization} allows us to prove the Lov\'asz Conjecture for
Ryser-extremal $3$-graphs.

\begin{proof}[Proof of Corollary~\ref{cor:lovasz}] 
By Theorem~\ref{thm:characterization} we have that $\cH$ has a home-base partition 
$({\cal F}, {\cal R}, W)$. Then by definition $|{\cal F}| + |{\cal R}| = \nu (\cH)$. 
Let us now define a pair of vertices from each element of ${\cal F} \cup {\cal R}$.
For each $F\in {\cal F}$ we take any of the partition classes $V_i$ and for 
each $R\in {\cal R}$ we take an arbitrary $2$-element subset $R' \subseteq R$. 
We claim that this system of $\nu(\cH)$ pairwise disjoint pairs of vertices satisfies the
statement of the theorem. 
To check this, suppose we  delete from ${\cal H}$ the union of an arbitrary $k$-set of these 
pairs, say the ones corresponding to 
some subfamilies ${\cal F}' \subseteq {\cal F}$ and ${\cal R}' \subseteq {\cal R}$. 
Consider a maximum matching $M$ in the remaining hypergraph. 
By the edge home property each edge $e$ of $M$ has a ``home'': $e$ is either contained in
some $F\in {\cal F}$ or it has two common vertices with some $R\in {\cal R}$. In either
case the home of $e$ cannot be from ${\cal F}' \cup {\cal R}'$ since  each pair of vertices 
deleted from these sets had a non-empty intersection with any edge of ${\cal H}$
that had its home there. 
Hence the edges of $M$ must have their home among the sets of 
$({\cal F} \setminus {\cal F}') \cup ({\cal R} \setminus {\cal R}')$. 
Since any two edges having the same home intersect, any set from
$({\cal F} \setminus {\cal F}') \cup ({\cal R} \setminus {\cal R}')$
can be home to at most one edge of $M$. 
Hence $|M| \leq  |{\cal F} \cup {\cal R}| -  |{\cal F}' \cup {\cal R}'| = 
\nu (\cH) - k$ and the claim is proved.
\end{proof}

\subsection{Proof Outline}
The main topic of our paper is the proof of Theorem~\ref{thm:characterization}. We have just seen that home-base hypergraphs are Ryser-extremal. The proof of the reverse implication will be done by induction on $\nu(\cH)$.

The case $\nu(\cH) = 0$ is trivial, and even the case $\nu(\cH) = 1$ is not difficult to check. Much of the work involved in proving the cases $\nu(\cH) \geq 2$ consists of finding an appropriate structure to which we can apply induction. That means a subhypergraph $\cH_0 \subseteq \cH$ which also satisfies $\tau(\cH_0) = 2\nu(\cH_0)$ and has $\nu(\cH_0) < \nu(\cH)$. By induction, this will have a home-base partition, but in order to be able to extend this partition to a home-base partition of the whole of $\cH$ we will also need the edges outside of $\cH_0$ to behave nicely.

A more precise description of the structure of the proof is given by the flow chart in Figure~\ref{fig:flowchart}. Please note that it is intended as a guide to be referred to throughout the proof, and many of the terms will only be introduced in later sections.

\begin{figure}
\begin{center}
\begin{tikzpicture}
[auto,
start/.style={rounded corners, draw, thick, fill=green!20},
decision/.style={diamond, aspect=2, draw, thick, fill=yellow!20, align=flush center, inner sep=1pt},
block/.style={rounded corners, draw, thick, fill=blue!20, text width=5em, align=center, minimum height=4em},
line/.style={draw, thick, -latex', rounded corners},
lemma/.style={draw, thick, ellipse, fill=red!20, align=center, minimum height=2em},
hidden/.style={inner sep=0}]

\node[lemma] (topology) at (8,1) {Topology};
\node[start] (start) at (0,.25) {$\tau(\cH) = 2\nu(\cH)$};
\node[lemma] (goodsets) at (1.75,-0.75) {L.~\ref{lem:goodsets}};
\node[lemma] (cpdecomposition) at (8,-0.75) {CP-Decom-\\position};
\node[decision] (goodset?) at (0,-2) {$\exists$ good set?};
\coordinate (goodsetsuse) at (1.75,-2);
\node[block, text width=8em] (link) at (4,-2) {Links have perf. matchings and all min. equinbrd. sets have size $2$};
\node[lemma] (extendtof) at (5.75,-4) {L.~\ref{prop:extendtof}};
\node[lemma] (goodtocromulent) at (-1.25,-5) {L.~\ref{lem:goodtocromulent}};
\coordinate (goodtocromulentuse) at (0,-5);
\node[decision, text width=11em] (disjointedges?) at (4,-5) {};
\node[align=center] (disjointedges?text) at (disjointedges?) {$\exists$ min.\\equinbrd. $X$ with $2$\\disj. edges?};
\coordinate (extendtofuse) at (6.25,-5);
\node[block, text width=6.7em] (fano) at (8,-5) {Min. equinbrd. sets extend to trunc. Fano planes};
\node[lemma] (perfectlycromulent) at (1.75,-7) {L.~\ref{lem:perfectlycromulent}};
\coordinate (twodisjointhyperedgesuse) at (4,-7);
\node[lemma] (twodisjointhyperedges) at (5.25,-7) {L.~\ref{lem:twodisjointhyperedges}};
\coordinate (endgameuse) at (8,-7);
\node[lemma] (endgame) at (9.25,-7) {L.~\ref{lem:endgame}};
\node[block] (pct) at (0,-8.5) {$\exists$ perfectly cromulent triple};
\coordinate (perfectlycromulentuse) at (1.75,-8.5);
\node[block, text width=6em] (ct) at (4,-8.5) {$\exists$ cromulent triple};
\node[block, text width=8em] (neighborhood) at (8,-8.5) {$\exists$ min. equinbrd. set $X$ with $X = N(N(X))$};
\node[lemma] (induction) at (0,-10) {L.~\ref{lem:induction}};
\coordinate (inductionuse) at (1.75,-10);
\node[block, text width=6em, fill=green!20] (homebase) at (4,-11) {$\cH$ is a home-base hypergraph};

\begin{scope}[every path/.style=line]
	\path (start) -- (goodset?);
	\path (goodset?) -- node[pos=0.05] {\textbf{YES}} (pct);
	\path (goodset?) -- node[very near start, below] {\textbf{NO}} (link);
	\path (pct) -- (homebase.west);
	\path (link) -- (disjointedges?);
	\path (disjointedges?) -- node[very near start] {\textbf{YES}} (ct);
	\path (disjointedges?) -- node[very near start, below] {\textbf{NO}} (fano);
	\path (ct) -- (pct);
	\path (fano) -- (neighborhood);
	\path (neighborhood) -- (homebase.east);
\end{scope}
\begin{scope}[every path/.style={line, dashed}]
	\path (goodsets) -- (goodsetsuse);
	\path (goodtocromulent) -- (goodtocromulentuse);
	\path (extendtof) -- (extendtofuse);
	\path (twodisjointhyperedges) -- (twodisjointhyperedgesuse);
	\path (endgame) -- (endgameuse);
	\path (induction) -- (inductionuse);
	\path (perfectlycromulent) -- (perfectlycromulentuse);
	\path (topology) -| (-1.75,-4) -- (goodtocromulent);
	\path (topology) -- (cpdecomposition);
	\path (topology) -- (goodsets);
	\path (topology) -| (9.75,-6) -- (endgame);
	\path (cpdecomposition) -- (goodsets);
	\path (cpdecomposition) -- (9.5,-4) -- (9.5,-6) -- (endgame);
\end{scope}
\end{tikzpicture}
\captionof{figure}{A flow-chart describing the logic of the proof with relevant lemmas shown.} \label{fig:flowchart}
\end{center}
\end{figure}
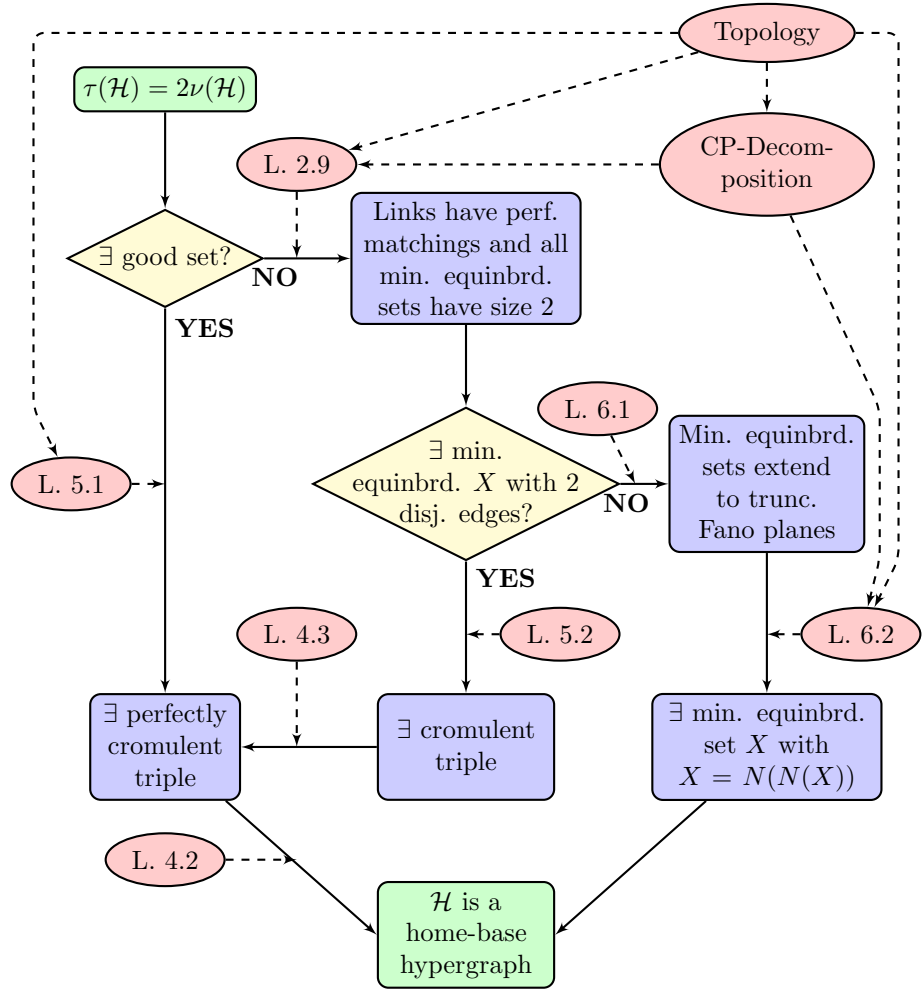

In Section~\ref{sec:graphtheorems}, we collect theorems we have shown in \cite{HNS} about the connectedness of the line graphs of the link graphs of Ryser-extremal $3$-graphs. Among others, this involves a structural characterization of the link graphs, which we call a CP-decomposition, as well as a theorem about bipartite graphs without so-called good sets. Good sets will turn out to be very useful to have in one of the link graphs of a Ryser-extremal $3$-graph, while the lack of good sets in a bipartite graph imposes very strong restrictions on its structure, which will eventually help us to show that we are dealing with a home-base hypergraph.

In Section~\ref{sec:homebasehypergraphs}, we prove some important properties of home-base hypergraphs, which will be essential for several parts of the rest of the proof.

In Section~\ref{sec:cromulenttriples}, we define and study cromulent and perfectly cromulent triples. A perfectly cromulent triple is a set of vertices such that the rest is a home-base hypergraph that interacts with the rest of the edges in a controlled fashion. This turns out to be precisely the substructure we need so that we can extend the home-base partition given by induction to a home-base partition of the whole hypergraph. Cromulent triples are apparently weaker versions of perfectly cromulent triples, but careful considerations will show that no cromulent triple can actually fail to be perfectly cromulent under the assumption that $\tau = 2\nu$. Therefore, it will be enough to find just a cromulent triple in order to show that we have a home-base hypergraph.

In Section~\ref{sec:searchingforcromulenttriples}, we show how to use a good set to find a perfectly cromulent triple and hence conclude that we are dealing with a home-base hypergraph. The rest of Section~\ref{sec:searchingforcromulenttriples} is devoted to exploring how the edges of the link graphs extend to hyperedges under the assumption that there are no good sets and no cromulent triples.

In Section~\ref{sec:theendgame}, we use the information on how the links extend, together with the fact that the links have CP-decompositions to show that the hypergraph must contain a truncated multi-Fano plane that interacts minimally with the rest of the hypergraph, which by induction will have a home-base partition. It is then easy to show that adding the lone $F$ results in a home-base partition of the whole hypergraph.

The proof of Theorem~\ref{thm:characterization} is assembled from all of the theorems and lemmas of the preceeding four sections in Section~\ref{sec:theproofoftheoremcharacterization}.

In Section~\ref{sec:concludingremarks} we prove a couple of facts related to our main theorem, some of them leading to interesting open questions.

\section{Theorems about the link graph} \label{sec:graphtheorems}

In this section we collect theorems that will be used in our arguments. For proofs and references, consult~\cite{HNS}.

The \emph{line graph} $L(\cH)$ of a hypergraph $\cH$ is the simple graph $L(\cH)$ on the vertex set $E(\cH)$ with  $e, f \in V(L(\cH))$ adjacent if $e \cap f \neq \emptyset$.

Recall that the \emph{connectedness} of a graph $G$, denoted $\conn(G)$, is the largest $k$ such that the independence complex of the graph $G$ is $k$-connected. The following theorem in this form is stated in \cite{HNS} as Theorem 1.2.

\begin{thm}{\cite{aharoniberger, aharonihaxell}} \label{thm:matchconn}
Let $\cG$ be an $r$-graph. Then
\[
	\conn(L(\cG)) \geq \frac{\nu(\cG)}{r} - 2.
\]
\end{thm}

\begin{defn} \label{def:linkg}
Let $\cH$ be a $3$-partite $3$-graph with parts $V_1$, $V_2$, and $V_3$. Let $S \subseteq V_i$ for some $i = 1, 2, 3$. Then the \emph{link graph} $\link{\cH}{S}$ is the bipartite graph with vertex classes $V_j$ and $V_k$ (where $\lset{i, j, k} = \lset{1, 2, 3}$) whose edge multiset is $\set{e \setminus V_i}{e \in E(\cH), e \cap V_i \subseteq S}$.
\end{defn}

\begin{prop}[{\cite[Proposition~2.1]{HNS}}] \label{prop:linkconn}
Let $\cH$ be a $3$-partite $3$-graph with vertex classes $V_1$, $V_2$, and $V_3$. 
Then for each $i \in \lset{1, 2, 3}$ we have the following:
\begin{enumerate}
	\renewcommand{\theenumi}{(\roman{enumi})}
	\renewcommand{\labelenumi}{\theenumi}
	\item \label{linkconn:lower} For all $S \subseteq V_i$ we have
		\[
			\conn(L(\link{\cH}{S})) \geq \frac{\tau(\cH) - (\abs{V_i} - \abs{S})}{2} - 2.
		\]
	\item \label{linkconn:upper} If $\nu({\cal H}) < |V_i|$, then there is some 
	$S \subseteq V_i$ such that
		\[
			\conn(L(\link{\cH}{S})) \leq \nu(\cH) - (\abs{V_i} - \abs{S}) - 2.
		\]
	\item \label{linkconn:size} If $\nu({\cal H}) < |V_i|$, then for every 
	$S \subseteq V_i$ for which the inequality in \ref{linkconn:upper} holds we have
		\[
			\abs{S} \geq \abs{V_i} - (2\nu(\cH) - \tau(\cH)).
		\]
\end{enumerate}
\end{prop}

\begin{thm}[{\cite[Theorem~1.4]{HNS}}] \label{thm:connoflink}
If $\cH$ is a $3$-partite $3$-graph with vertex classes $V_1$, $V_2$, and $V_3$, such that $\tau(\cH) = 2\nu(\cH)$, then for each $i$ we have
\begin{enumerate}
	\renewcommand{\theenumi}{(\roman{enumi})}
	\renewcommand{\labelenumi}{\theenumi}
	\item \label{connoflink:conn} $\conn(L(\link{\cH}{V_i})) = \nu(\cH) - 2$.
	\item \label{connoflink:nu} $\nu(\link{\cH}{V_i}) = \tau(\cH)$.
\end{enumerate}
In particular 
\begin{equation}\label{eq:linkconn}
\conn(L(\link{\cH}{V_i})) = \frac{\nu(\link{\cH}{V_i})}{2}  - 2. 
\end{equation} 
\end{thm}

\begin{thm}[{\cite[Theorem~1.5]{HNS}}] \label{thm:cpdecomposition}
Let $G$ be a bipartite graph. Then we have $\conn(L(G)) = \frac{\nu(G)}{2} - 2$ if and only if $G$ has a collection of $\nu(G)/2$ pairwise vertex-disjoint subgraphs, each of them a $C_4$ or a $P_4$, such that every edge of $G$ is parallel to an edge of one of the $C_4$'s or is incident to an interior vertex of one of the $P_4$'s.
\end{thm}

We refer to such a collection as a \emph{CP-decomposition}. Note that this is just a specialization of the concept of CP-decomposition in~\cite{HNS} for the entire line graph, which is the only case we will need in this paper. As promised in~\cite{HNS}, the ``if'' direction of this theorem will be proved in this paper. We will postpone the proof until Section~\ref{sec:concludingremarks}, as it is not necessary for the proof of the main theorem.

For a subset $X$ of the vertices of a graph, we denote the \emph{neighborhood} of $X$ by $N(X)$, meaning the set of vertices adjacent to some vertex in $X$.

\begin{defn} \label{def:decent}
Let $G$ be a bipartite graph with vertex classes $A$ and $B$. A subset $X \subseteq B$ is called \emph{decent} if it satisfies the following conditions:
\begin{enumerate}
	\renewcommand{\theenumi}{(\arabic{enumi})}
	\renewcommand{\labelenumi}{\theenumi}
	\item \label{decent:neighborhood} $\abs{N(X)} \leq \abs{X}$,
	\item \label{decent:nu} $\nu(G) = \abs{N(X)} + \abs{B \setminus X}$,
	\item \label{decent:matchings} For every $x \in X$ and $y\in N(x)$ the edge $xy$ participates in a maximum matching of $G$.
\end{enumerate}
\end{defn}

\begin{defn}
Let $G$ be a bipartite graph. A subset $X$ of a vertex class of $G$ is called \emph{equineighbored} if $X$ is nonempty and $\abs{N(X)} = \abs{X}$.
\end{defn}

\begin{defn} \label{def:good}
Let $G$ be a bipartite graph with vertex classes $A$ and $B$. A subset $X \subseteq B$ is called \emph{good} if it is decent, and if for all $y \in N(X)$ we have $\conn \left( L \left( G - \set{yz \in E(G)}{z \in B \setminus X} \right) \right) > \conn(L(G))$.
\end{defn}

Note in particular that if $X$ is good, then $\set{yz \in E(G)}{z \in B \setminus X} \neq \emptyset$ for all $y \in N(X)$.

\begin{lem}[{\cite[Lemma~4.7]{HNS}}] \label{lem:goodsets}
Let $G$ be a bipartite graph with vertex classes $A$ and $B$. Suppose $\nu(G) = 2k$ for some integer $k$ and $\conn(L(G)) = k - 2$. If $G$ has no good set in $A$ nor in $B$, then the following hold:
\begin{enumerate}
	\renewcommand{\theenumi}{(\roman{enumi})}
	\renewcommand{\labelenumi}{\theenumi}
	\item \label{goodsets:perfect} $G$ has a perfect matching
	\item \label{goodsets:minimal} For every minimal equineighbored subset $X \subseteq A$ or $X \subseteq B$ we have $\abs{X} = 2$. In particular, $G[X \cup N(X)]$ is a $C_4$ (possibly with parallel edges).
\end{enumerate}
\end{lem}

Note that the minimality requirement in~\ref{goodsets:minimal} is well-defined because by~\ref{goodsets:perfect} both $A$ and $B$ are equineighbored.

\section{Properties of Home-Base Hypergraphs} \label{sec:homebasehypergraphs}

The next couple of sections will establish some basic properties of home-base hypergraphs that we will need in the proof of Theorem~\ref{thm:characterization}.

First is the so-called ``monster lemma,'' which states under which conditions a monster can eat some vertices of a home-base hypergraph without reducing the matching number.

But before we can prove it, we shall need some definitions.

\subsection{Essential and Superfluous Vertices}

\begin{defn}
Let $G$ be a bipartite graph with vertex classes $X_1$ and $X_2$. A subset $C \subseteq X_i$ is called \emph{essential} if there is a subset $U \subseteq X_{3 - i}$ with $\abs{U} = \abs{C}$ and $C = N(U)$.
\end{defn}

We remark briefly that non-empty essential subsets are precisely the neighborhoods of equineighbored subsets. We will of course apply this concept to the bipartite graphs $B_i$ from the matchability criterion of FR-partitions.

Let $\cH$ be a $3$-partite $3$-graph on vertex classes $V_1$, $V_2$, and $V_3$ with a matchable FR-partition $(\cF, \cR, W)$. We call a vertex $v$ in $V_i$ essential if $v \in W$ and $\lset{v} \subseteq W \cap V_i$ is essential in $B_i$. If $R \in \cR$ has only $v \in W \cap V_i$ as its neighbor in $B_i$, then we say $v$ is essential for $R$.

\begin{lem} \label{lem:essential}
Let $B$ be a bipartite graph with vertex classes $\cR$ and $W$, which has a matching saturating $\cR$. Then $W$ contains a unique maximal essential subset.
\end{lem}

\begin{proof}
Let $C_1, C_2 \subseteq W$ be essential. Then we claim $C_1 \cup C_2$ is also essential. Consider $\cU_1, \cU_2 \subseteq \cR$ such that $C_1 = N_B(\cU_1)$, $C_2 = N_B(\cU_2)$, $\abs{\cU_1} = \abs{C_1}$ and $\abs{\cU_2} = \abs{C_2}$. Then $N_B(\cU_1 \cup \cU_2) = C_1 \cup C_2$ and by Hall's Theorem, $\abs{C_1 \cup C_2} \geq \abs{\cU_1 \cup \cU_2}$. But of course $N_B(\cU_1 \cap \cU_2) \subseteq C_1 \cap C_2$ and thus again by Hall's Theorem, $\abs{C_1 \cap C_2} \geq \abs{\cU_1 \cap \cU_2}$. By the inclusion-exclusion principle, we thus have $\abs{C_1} + \abs{C_2} - \abs{C_1 \cup C_2} \geq \abs{\cU_1} + \abs{\cU_2} - \abs{\cU_1 \cup \cU_2}$, and since $\abs{\cU_1} = \abs{C_1}$ and $\abs{\cU_2} = \abs{C_2}$, we find that $\abs{C_1 \cup C_2} \leq \abs{\cU_1 \cup \cU_2}$, so that in fact there is equality. This proves that $C_1 \cup C_2$ is essential. Therefore the union over all essential subsets of $W$ gives the unique maximal essential set.
\end{proof}

A vertex of $W$ which is not in the maximal essential set is called \emph{superfluous}. Note that any one superfluous vertex can be removed, and the rest of the bipartite graph will still have a matching saturating $\cR$. Again, we will apply this to the bipartite graphs $B_i$ from the matchability criterion of FR-partitions.

Let $\cH$ be a home-base hypergraph on vertex classes $V_1$, $V_2$, and $V_3$ with a home-base partition $(\cF, \cR, W)$. Then the auxiliary bipartite graphs $B_i$ have vertex classes $\cR$ and $W \cap V_i$ and a matching saturating $\cR$. Therefore, each $W \cap V_i$ contains a unique maximum essential subset $C_i$, and we may call a vertex of $V_i$ superfluous if it is in $W \cap V_i \setminus C_i$. Clearly superfluous vertices are non-essential $W$-vertices in a stronger form. We can make the following observation:

\begin{obs} \label{obs:superfluous}
Let $\cH$ be a $3$-partite $3$-graph with a matchable FR-partition $(\cF, \cR, W)$, and let $S \subseteq W$ be a set of superfluous vertices with at most one vertex in each vertex class. Then $(\cF, \cR, W \setminus S)$ is a matchable FR-partition of $\cH - S$.
\end{obs}

\begin{proof}
Since removing any single superfluous vertex $s$ from any of the bipartite graphs $B_i$ leaves a matching saturating $\cR$, $(\cF, \cR, W \setminus \lset{s})$ is a matchable FR-partition. Since removing $s$ from one does not change the other graphs $B_j$ at all, we can do this for each vertex class independently.
\end{proof}

We will need the following simple lemma about removing superfluous vertices later in Section~\ref{sec:searchingforcromulenttriples}.

\begin{lem} \label{lem:superfluousessential}
Let $B$ be a bipartite graph with vertex classes $\cR$ and $W$ that has a matching saturating $\cR$, and let $C \subseteq W$ be the maximal essential subset. If $p \in C$ and $s \in W \setminus C$, then $p$ is essential in $B$ if and only if it is essential in $B - s$.
\end{lem}

\begin{proof}
If $p$ is essential in $B$, then it clearly is essential in $B - s$.

Conversely, assume $p$ is essential in $B - s$. Let $\cU \subseteq \cR$ be such that $N_B(\cU) = C$ and $\abs{\cU} = \abs{C}$, which exists by the definition of essential subsets. Since $p$ is essential, there is a unique $R \in \cR$ such that $N_{B - s}(R) = \lset{p}$. We claim that $R \in \cU$. Suppose not. Then $N_B(R) \subseteq \lset{s, p}$, and hence $N_B(\cU \cup \lset{R}) \subseteq C \cup \lset{s}$. Since $\abs{\cU \cup \lset{R}} = \abs{\cU} + 1 = \abs{C \cup \lset{s}}$, this would make $C \cup \lset{s}$ an essential set in $B$, a contradiction, since $C$ is maximal. Hence $R \in \cU$, from which follows that $s \notin N_B(R)$, and thus $N_B(R) = \lset{p}$, so $p$ is essential in $B$.
\end{proof}

\subsection{The Monster Lemma} \label{sec:monsterlemmas}

\begin{lem} \label{lem:monster}
Let $\cH$ be a $3$-partite $3$-graph that has a matchable FR-partition $(\cF, \cR, W)$. Let $a, b, c \in V(\cH)$ be in different vertex classes. Suppose that the following two conditions hold:
\begin{enumerate}
	\renewcommand{\theenumi}{(\arabic{enumi})}
	\renewcommand{\labelenumi}{\theenumi}
	\item \label{monster:fedge} For every $F \in \cF$, there is an $F$-edge avoiding $\lset{a, b, c}$,
	\item \label{monster:redge} For every $R \in \cR$, there is an $R$-edge avoiding $\lset{a, b, c}$.
\end{enumerate}
Then $\nu(\cH - \lset{a, b, c}) = \nu(\cH)$.
\end{lem}

\begin{proof}
Let $V_1$, $V_2$, and $V_3$ be the vertex classes of $\cH$, where $a \in V_1$, $b \in V_2$, and $c \in V_3$. We will select a matching $\cM \subseteq E(\cH)$ of size $\nu(\cH)$ avoiding $\lset{a, b, c}$.

First, for each $F \in \cF$ we choose an arbitrary edge from $\cH|_F$ avoiding $\lset{a, b, c}$ and include it in $\cM$. This can be done by condition~\ref{monster:fedge}. These edges are all pairwise disjoint, since the members of $\cF$ are pairwise disjoint. Furthermore, we will describe a procedure that selects pairwise disjoint $\cR$-edges, one for each $R \in \cR$, each containing a $W$-vertex and avoiding $\lset{a, b, c}$. Because they contain a $W$-vertex, these $\cR$-edges will all be disjoint from the $\cF$-edges we already put into $\cM$ (since both $W$ and $V(\cR)$ are disjoint from $V(\cF)$). If successful, we will have constructed the required matching $\cM$, since $\abs{\cM} = \abs{\cF} + \abs{\cR} = \nu(\cH)$.

How we choose the $\cR$-edges will fall into several cases. We introduce the following convenient notation for talking about $\cR$-edges. An $\cR$-edge $xyz$ of $\cH$ is called a \emph{WRR-edge} if $x \in W \cap V_1$. Analogously, $xyz$ is called an \emph{RWR-edge} or an \emph{RRW-edge} if $y \in W \cap V_2$ or $z \in W \cap V_3$, respectively.

\noindent \textbf{Case 1}. At least one of the vertices $a$, $b$, or $c$ is in $V(\cR)$.

We may assume without loss of generality that $a \in V(\cR)$. First we choose a matching $M_1$ saturating $\cR$ in the auxiliary bipartite graph $B_1$. Such a matching exists by the matchability of the FR-partition. Each edge $Rw \in M_1$, with $R \in \cR$ and $w \in W \cap V_1$ corresponds to a WRR-edge of $\cH$ consisting of $w$ and two vertices of $R$. These edges form a matching $\cM'$ of $\cR$-edges in $\cH$. Each edge in $\cM'$ contains a $W$-vertex in $V_1$ and hence avoids $a \in V(\cR) \cap V_1$. The only problem might be that $b$ or $c$ appear in some of these edges, rendering those edges unsuitable. If $b$ is contained in the $R$-edge $e_1 \in \cM'$ for some $R \in \cR$, then replace $e_1$ in $\cM'$ with an arbitrary RWR-edge $e_2$ for $R$. Such an edge exists because $B_2$ has a matching saturating $\cR$, and it is disjoint from all other edges in $\cM'$ because these are WRR-edges. The vertex of $e_2$ in $V_1$ cannot be $a$, since then all $R$-edges would intersect $\lset{a, b}$, contradicting condition~\ref{monster:redge}. Similarly, the vertex of $e_2$ in $V_3$ cannot be $c$, since then all $R$-edges would intersect $\lset{b, c}$. Finally, if $c$ is contained in the $R'$-edge $e_3 \in \cM'$ for some $R' \in \cR$, then replace $e_3$ in $\cM'$ with an arbitrary RRW-edge $e_4$ for $R'$. Such an edge exists because $B_3$ has a matching saturating $\cR$, and it is disjoint from all other edges of $\cM'$ because they are all WRR- and RWR-edges. The edge $e_4$ cannot contain $a$, otherwise all $R'$-edges would intersect $\lset{a, c}$, contradicting~\ref{monster:redge}. The edge $e_4$ also does not contain $b$, since otherwise every $R'$-edge would intersect $\lset{b, c}$, again contradicting~\ref{monster:redge}.

Now the vertices of the matching $\cM'$ avoid $\lset{a, b, c}$ and Case 1 is complete.

Let us assume from now on that none of the vertices $a$, $b$, and $c$ are in $V(\cR)$.

\noindent \textbf{Case 2}. None of the vertices $a$, $b$, and $c$ are essential.

First we choose a matching $M_1$ in $B_1$ saturating $\cR$, which exists by the matchability of the FR-partition. This corresponds to a matching $\cM'$ in $\cH$ consisting of WRR-edges. Clearly, $b$ and $c$ are avoided by the edges of $\cM'$ because $b, c \notin V(\cR)$. If $a$ is contained in an $R$-edge $e_1 \in \cM'$ for some $R \in \cR$, then replace $e_1$ in $\cM'$ by an arbitrary RWR-edge $e_2$ for $R$ that avoids $b$. This can be done, since $b$ is not essential. The edge $e_2$ also avoids $a$ and $c$ because $a, c \notin V(\cR)$, and it is disjoint from all other edges of $\cM'$ because they are all WRR-edges.

Hence we have the required matching $\cM'$ avoiding $\lset{a, b, c}$ and Case 2 is complete.

\noindent \textbf{Case 3}. Not all of the vertices $a$, $b$, and $c$ are essential $W$-vertices for the same $R \in \cR$.

We may assume without loss of generality that $a$ is essential for $R \in \cR$ (If no vertex is essential, we are in Case 2). By assumption, not both $b$ and $c$ are essential for $R$ as well, so assume without loss of generality that $b$ is not essential for $R$. We choose a matching $M_1 \subseteq E(B_1)$ saturating $\cR$. This corresponds to a matching $\cM'$ in $\cH$ consisting of WRR-edges. Clearly, $b$ and $c$ are avoided by the edges of $\cM'$ because $b, c \notin V(\cR)$. Since $a$ is essential for $R$, it must be that $Ra \in M_1$ because $a$ is the only neighbor of $R$ in $W \cap V_1$. Let $e_1 \in \cM'$ be the edge corresponding to $Ra \in M_1$. We replace $e_1$ in $\cM'$ by an arbitrary RWR-edge $e_2$ for $R$ that avoids $b$. This can be done, since $b$ is not essential for $R$. The edge $e_2$ also avoids $a$ and $c$ because $a, c \notin V(\cR)$, and it is disjoint from all other edges of $\cM'$ because they are all WRR-edges.

This means that $\cM'$ avoids $\lset{a, b, c}$, and so Case 3 is complete.

\noindent \textbf{Case 4}. The vertices $a$, $b$, and $c$ are all essential $W$-vertices for $R \in \cR$.

By condition~\ref{monster:redge}, there must be an $R$-edge $e$ avoiding $a$, $b$, and $c$. At least two of its vertices must be in $R$, so assume without loss of generality that $e \cap V_2, e \cap V_3 \subseteq R$. We choose a matching $M_1$ in $B_1$ saturating $\cR$. It corresponds to a matching $\cM'$ of WRR-edges in $\cH$. Because $a$ is essential for $R$, it follows that there is an edge of $\cM'$ containing $a$ and two vertices of $R$. Replace it by $e$, which avoids $a$, $b$, and $c$ and is disjoint from the other edges of $\cM'$ because its $V_1$-vertex is not in $W$ (because $a$ is the only $W$-vertex in a WRR-edge of $R$) and its other vertices are in $R$. The rest of the edges of $\cM'$ clearly avoid $a$, $b$, and $c$, since the one edge of $\cM'$ containing $a$ has already been replaced, and $b, c \notin V(\cR)$.

We must be careful because in this case, one of the edges of $\cM'$, namely $e$, is not necessarily contained in $V(\cR) \cup W$, as has been true in all other cases. Thus, the $V_1$-vertex of $e$ may be in some $F \in \cF$, and hence could potentially intersect the $F$-edge which we added to $\cM$ in the beginning. However, since $\cH|_F$ is a truncated multi-Fano plane, it cannot be covered by one vertex, so there is an $F$-edge disjoint from $e$ with which we can replace our original choice of edge for $\cM$. Note that we do not need to worry about avoiding $\lset{a, b, c}$ with this edge, as these are all in $W$.

Adding the edges in $\cM'$ to $\cM$ gives us our desired matching avoiding $\lset{a, b, c}$. This concludes Case 4.

These cases exhaust all possibilities, so the proof is complete.
\end{proof}

In order to facilitate the use of this lemma, we prove in some specific cases that the conditions are fulfilled.

\begin{cor} \label{cor:monstercor}
Let $\cH$ be a $3$-partite $3$-graph with a matchable FR-partition $(\cF, \cR, W)$. Let $a, b, c \in V(\cH)$ be in different vertex classes, and let $S \subseteq W$ be a set of superfluous vertices with at most one vertex in each vertex class. Then in any of the following cases we have $\nu(\cH - (\lset{a, b, c} \cup S)) = \nu(\cH)$:
\begin{enumerate}
	\renewcommand{\theenumi}{(\arabic{enumi})}
	\renewcommand{\labelenumi}{\theenumi}
	\item \label{monstercor:fwx} $a \in V(\cF)$, $b \in W$, and $c$ is arbitrary,
	\item \label{monstercor:rnrnr} $a \in R \in \cR$, $b \notin R$, and $c \notin V(\cR)$,
	\item \label{monstercor:enenr} $a \in W$ is essential for $R \in \cR$, $b$ is not essential for $R$ in $\cH - S$, and $c \notin V(\cR)$,
	\item \label{monstercor:nenrx} $a \in W$ is not essential in $\cH - S$, $b \notin V(\cR)$, and $c$ is arbitrary.
\end{enumerate}
\end{cor}

\begin{proof}
Let $V_1$, $V_2$, and $V_3$ be the vertex classes of $\cH$, where $a \in V_1$, $b \in V_2$, and $c \in V_3$. Let $S' = S \setminus \lset{a, b, c}$. By Observation~\ref{obs:superfluous}, the hypergraph $\cH' = \cH - S'$ has the matchable FR-partition $(\cF, \cR, W \setminus S')$, and hence $\nu(\cH') = \nu(\cH)$. We will apply Lemma~\ref{lem:monster} to $\cH'$ to find a matching in $\cH'$ of size $\nu(\cH')$ avoiding $\lset{a, b, c}$. This constitutes a matching in $\cH - (\lset{a, b, c} \cup S)$ of size $\nu(\cH)$, as desired. We must simply check that the two conditions of Lemma~\ref{lem:monster} hold.

\noindent \textbf{Case 1}. $a \in V(\cF)$, $b \in W$, and $c$ is arbitrary.

For any $F \in \cF$, there is an $F$-edge avoiding $\lset{a, b, c}$, because $b \in W$, and $a$ and $c$, being in different vertex classes, do not cover every edge of $\cH'|_F$ (a truncated multi-Fano plane).

Let $R = \lset{r_1, r_2, r_3} \in \cR$ (where $r_i \in V_i$). We will find an $R$-edge avoiding $\lset{a, b, c}$. If $c \in R$, then there is an $R$-edge avoiding $\lset{a, b, c}$ because the matchability of $B_3$ ensures that there is an $R$-edge $r_1 r_2 w$ with $w \in W \cap V_3$, which clearly avoids $\lset{a, b, c}$, because $a, b \notin V(\cR)$, and $c \in R$. Suppose $c \notin R$. By the matchability of $B_1$, there is an $R$-edge $w'r_2 r_3$, where $w' \in W \cap V_1$, and this edge avoids $\lset{a, b, c}$ because $a \in V(\cF)$, $b \in W$, and $c \notin R$.

Therefore Lemma~\ref{lem:monster} applies, and we have $\nu(\cH' - \lset{a, b, c}) = \nu(\cH)$.

\noindent \textbf{Case 2}. $a \in R \in \cR$, $b \notin R$, and $c \notin V(\cR)$.

For any $F \in \cF$, there is an $F$-edge avoiding $\lset{a, b, c}$, because $a \in V(\cR)$, and $b$ and $c$ do not cover every edge of $\cH'|_F$ (a truncated multi-Fano plane).

Let $R' = \lset{r_1, r_2, r_3} \in \cR$ (where $r_i \in V_i$). We will find an $R'$-edge avoiding $\lset{a, b, c}$. If $b \in R'$, then $R' \neq R$, so $a \notin R'$. There is an $R'$-edge $r_1 wr_3$ with $w \in W \cap V_2$ by matchability applied to $B_2$. This edge avoids $\lset{a, b, c}$ because $a \notin R'$, $b \in R'$, and $c \notin V(\cR)$. Suppose $b \notin R'$. By the matchability of $B_1$, there is an $R'$-edge $w'r_2 r_3$, where $w' \in W \cap V_1$, and this edge avoids $\lset{a, b, c}$ because $a \in V(\cR)$, $b \notin R'$, and $c \notin V(\cR)$.

Therefore Lemma~\ref{lem:monster} applies, and we have $\nu(\cH' - \lset{a, b, c}) = \nu(\cH)$.

\noindent \textbf{Case 3}. $a \in W$ is essential for $R \in \cR$, $b$ is not essential for $R$ in $\cH - S$, and $c \notin V(\cR)$.

Note that if $a$ is essential for $R$ in $\cH$, then it is still essential for $R$ in $\cH'$, a subgraph of $\cH$. Similarly, if $b$ is not essential for $R$ in $\cH - S$, then it certainly is not essential for $R$ in $\cH'$, since $\cH - S$ is a subhypergraph of $\cH'$.

For any $F \in \cF$, there is an $F$-edge avoiding $\lset{a, b, c}$, because $a \in W$, and $b$ and $c$ do not cover every edge of $\cH'|_F$ (a truncated multi-Fano plane).

Let $R' = \lset{r_1, r_2, r_3} \in \cR$ (where $r_i \in V_i$). We will find an $R'$-edge avoiding $\lset{a, b, c}$. If $b$ is not essential for $R'$, then $R'$ has a neighbor $w \in W \cap V_1$ in $B_2$ with $w \neq b$. The $R'$ edge $r_1 wr_3$ then avoids $\lset{a, b, c}$ because $a \in W$, $b \neq w$, and $c \notin V(\cR)$. If $b$ is essential for $R'$, then $b \in W$ and $R' \neq R$, so $a$ is not essential for $R'$ (because no vertex can be essential for two different members of $\cR$ by matchability). Thus $R'$ has a neighbor $w' \in W \cap V_1$ in $B_1$ with $w' \neq a$. The $R'$-edge $w'r_2 r_3$ then avoids $\lset{a, b, c}$ because $w' \neq a$ and $b, c \notin V(\cR)$.

Therefore Lemma~\ref{lem:monster} applies, and we have $\nu(\cH' - \lset{a, b, c}) = \nu(\cH)$.

\noindent \textbf{Case 4}. $a \in W$ is not essential in $\cH - S$, $b \notin V(\cR)$, and $c$ is arbitrary.

Note that if $a$ is not essential in $\cH - S$, then it certainly is not essential in $\cH'$, since $\cH - S$ is a subhypergraph of $\cH'$.

For any $F \in \cF$, there is an $F$-edge avoiding $\lset{a, b, c}$, because $a \in W$, and $b$ and $c$ do not cover every edge of $\cH'|_F$ (a truncated multi-Fano plane).

Let $R = \lset{r_1, r_2, r_3} \in \cR$ (where $r_i \in V_i$). We will find an $R$-edge avoiding $\lset{a, b, c}$. If $c \in R$, then there is an $R$-edge avoiding $\lset{a, b, c}$ because the matchability of $B_3$ ensures that there is an $R$-edge $r_1 r_2 w$ with $w \in W \cap V_3$, which clearly avoids $\lset{a, b, c}$, since $a, b \notin V(\cR)$, and $c \in R$. Suppose $c \notin R$. Since $a$ is not essential, $R$ has a neighbor $w' \in W \cap V_1$ in $B_1$ with $w' \neq a$. The $R$-edge $w'r_2 r_3$ then avoids $\lset{a, b, c}$ because $w' \neq a$, $b \notin V(\cR)$, and $c \notin R$.

Therefore Lemma~\ref{lem:monster} applies, and we have $\nu(\cH' - \lset{a, b, c}) = \nu(\cH)$.
\end{proof}

It is unfortunately necessary in Cases 3 and 4 to make sure that the non-essential $W$-vertex remains non-essential after removing the superfluous vertices. However, this condition is often very easy to check, since removing superfluous vertices from the hypergraph only affects the status of those $W$-vertices in their vertex class. This leads to the following observation:

\begin{obs}
Let $\cH$ be a $3$-partite $3$-graph with a matchable FR-partition $(\cF, \cR, W)$, and let $s \in W$ be a superfluous vertex. Then if $w \in W$ is in a different vertex class from $s$, it holds that $w$ is non-essential in $\cH$ if and only if it is non-essential in $\cH - s$.
\end{obs}

\subsection{Matchability and the Edge-Home Property}

One nice consequence of the monster lemma is the following proposition, which will be key to our proof.

\begin{defn}
An FR-partition $(\cF, \cR, W)$ is \emph{proper} if there is no $R \in \cR$ and an edge of $\cH$ consisting of three vertices of $W$ which together induce a truncated Fano plane. Being proper just means that we have not called anything an $R$ if it could have been part of an $F$.
\end{defn}

Clearly home-base partitions are proper, because they do not contain any edges consisting of $W$-vertices. It turns out that a converse to this fact is also true.

\begin{prop} \label{prop:matchableedgehome}
A proper matchable FR-partition of a $3$-partite $3$-graph has the edge-home property.
\end{prop}

\begin{proof}
Let $\cH$ be a $3$-partite $3$-graph with vertex classes $V_1$, $V_2$, $V_3$, and let $(\cF, \cR, W)$ be a proper matchable FR-partition of $\cH$. Let $abc$ be an edge of $\cH$. We aim to show that it is either an $\cF$-edge or an $\cR$-edge. Suppose it is not. We will aim for a contradiction by applying Lemma~\ref{lem:monster} to show $\cH - \lset{a, b, c}$ has a matching of size $\nu(\cH)$.

By assumption, $abc$ is not in $\cH|_F$ for any $F \in \cF$, which means that every $F \in \cF$ has an $F$-edge avoiding $\lset{a, b, c}$, since the only way to cover a truncated Fano plane with vertices from different vertex classes is if they form one of its edges. We want to show that it also cannot cover every $R$-edge for any $R \in \cR$.

Since the partition is matchable, each of the auxiliary bipartite graphs $B_1$, $B_2$, and $B_3$ have matchings saturating $\cR$, say $M_1$, $M_2$, and $M_3$, respectively. Then each $R = \lset{r_1, r_2, r_3} \in \cR$ has three $W$-vertices, $w^R_i \in V_i$ assigned to it, so that $Rw^R_i \in M_i$, which means that $w^R_i r_j r_k$ are edges for each choice of $\lset{i, j, k} = \lset{1, 2, 3}$. By assumption, $abc$ intersects $R$ in at most one vertex (otherwise, it is an $R$-edge). If $abc$ intersects $R$ in one vertex, without loss of generality in $V_1$, then $w^R_1 r_2 r_3$ is an $R$-edge disjoint from $abc$. If $abc$ does not intersect $R$ in any vertex, then it intersects all the $R$-edges $w^R_i r_j r_k$ for $\lset{i, j, k} = \lset{1, 2, 3}$ only if $abc = w^R_1 w^R_2 w^R_3$, which would mean that $abc$, $w^R_1 r_2 r_3$, $r_1 w^R_2 r_3$, and $r_1 r_2 w^R_3$ form a truncated Fano plane. If this is the case, then we claim that these are in fact the only edges on $\lset{a, b, c, r_1, r_2, r_3}$, which would contradict the assumption that $(\cF, \cR, W)$ is proper.

Suppose these are not the only edges on $\lset{a, b, c, r_1, r_2, r_3}$. Then there are two disjoint edges on $\lset{a, b, c, r_1, r_2, r_3}$. Now pick one $F$-edge for each $F \in \cF$, and take the edges $w^{R'}_1 r'_2 r'_3$ for each $R' \in \cR \setminus \lset{R}$. These edges form a matching of size $\abs{\cF} + \abs{\cR} - 1$, and they do not intersect $\lset{a, b, c, r_1, r_2, r_3}$. Together with the two disjoint edges on $\lset{a, b, c, r_1, r_2, r_3}$, we find a matching of size $\abs{\cF} + \abs{\cR} + 1 = \nu(\cH) + 1$, a contradiction.

Hence $a$, $b$, and $c$ fulfill the conditions of Lemma~\ref{lem:monster}, and $\cH \setminus \lset{a, b, c}$ would have a matching of size $\nu(\cH)$, which together with $abc$ would be a matching of size $\nu(\cH) + 1$ in $\cH$, a contradiction. Therefore $\cH$ has the edge-home property.
\end{proof}

\section{Cromulent Triples}  \label{sec:cromulenttriples}

The aim of this section is to define the appropriate substructure which will facilitate the inductive proof of our main theorem (Theorem~\ref{thm:characterization}). The key definition is that of a cromulent triple.

\begin{defn}
Let $\cH$ be a $3$-partite $3$-graph with vertex classes $V_1$, $V_2$, and $V_3$. A triple of nonempty sets $(Y_1, Y_2, X)$ with $Y_1 \subseteq V_i$, $Y_2 \subseteq V_j$ and $X \subseteq V_k$, where $\lset{i, j, k} = \lset{1, 2, 3}$ is called a \emph{cromulent} triple if it fulfills the following conditions:
\begin{enumerate}
	\renewcommand{\theenumi}{(\arabic{enumi})}
	\renewcommand{\labelenumi}{\theenumi}
	\item \label{cromulent:size} $\abs{Y_1} = \abs{Y_2} \leq \abs{X}$,
	\item \label{cromulent:neighborhood} $N_{\link{\cH}{V_i}}(X) = Y_2$,
	\item \label{cromulent:matching} There is a hypergraph matching in $\cH|_{Y_1 \cup Y_2 \cup X}$ of size $\abs{Y_1}$,
	\item \label{cromulent:homebase} The hypergraph $\cH_0 = \cH - (Y_1 \cup Y_2 \cup X)$ is a home-base hypergraph with $\nu(\cH_0) = \nu(\cH) - \abs{Y_1}$,
	\item \label{cromulent:edges} Given any home-base partition $(\cF, \cR, W)$ of $\cH_0$, we have $N_{\link{\cH}{V_j}}(X) \subseteq Y_1 \cup V(\cR) \cup V(\cF)$.
\end{enumerate}
Such a triple is called \emph{perfectly cromulent} if it fulfills the following stronger version of condition~\ref{cromulent:edges}:
\begin{enumerate}
	\item[(5*)] \label{cromulent:perfect} $N_{\link{\cH}{V_j}}(X) = Y_1$.
\end{enumerate}
\end{defn}

The first lemma of this section states that perfectly cromulent triples are the kind of substructure we should look for in order to prove our main theorem.

\begin{lem} \label{lem:induction}
Let $\cH$ be a $3$-partite $3$-graph with $\tau(\cH) = 2\nu(\cH)$. If $\cH$ has a perfectly cromulent triple, then $\cH$ is a home-base hypergraph.
\end{lem}

Unfortunately, it is sometimes hard to ensure property~(5*), and it will be easier to find just cromulent triples instead. Fortunately, we will be able to prove that this suffices.

\begin{lem} \label{lem:perfectlycromulent}
If $\cH$ is a $3$-partite $3$-graph with $\tau(\cH) = 2\nu(\cH)$, then every cromulent triple of $\cH$ is perfectly cromulent.
\end{lem}

These two lemmas combine to give the main result of this section as an immediate corollary:

\begin{cor} \label{cor:cromulenttohbh}
Let $\cH$ be a $3$-partite $3$-graph with $\tau(\cH) = 2\nu(\cH)$. If $\cH$ has a cromulent triple, then $\cH$ is a home-base hypergraph.
\end{cor}

The proofs of the two lemmas follow similar lines, and so they will be handled in parallel. The basic idea is outlined below. We start with Lemma~\ref{lem:induction}.

Let $(Y_1, Y_2, X)$ be a perfectly cromulent triple, and let $\cH_0 = \cH - (Y_1 \cup Y_2 \cup X)$ be the hypergraph from the definition of cromulent triples. Let $(\cF, \cR, W)$ be a home-base partition of $\cH_0$. Our goal will be to extend this partition into a home-base partition $(\cF', \cR', W')$ of $\cH$. Fix a maximum hypergraph matching $\cM$ in $\cH|_{Y_1 \cup Y_2 \cup X}$. Each pair $y \in Y_1$, $y' \in Y_2$ that are together in an edge of $\cM$ will participate in a new $R \in \cR'$ together with a uniquely determined member of $W \cap V_3$. The vertices in $X$ will be vertices of $W'$, and by virtue of the matching saturating $Y_1$ and $Y_2$, they will ensure a matching saturating $\cR'$ exists in the bipartite graph $B_3'$. The rest of the section will be devoted to finding the member of $W \cap V_3$ we can include in our new $R$'s and proving that the resulting partition $(\cF', \cR', W')$ is indeed a home-base partition. Our fundamental tool in this proof will be Corollary~\ref{cor:monstercor}, and we will finish by using Proposition~\ref{prop:matchableedgehome}.

If $(Y_1, Y_2, X)$ was simply a cromulent triple, then much of the same proof as above still goes through in a more restricted form, and eventually we will be able to find a contradiction if $(Y_1, Y_2, X)$ violated condition~(5*), which will show Lemma~\ref{lem:perfectlycromulent}.

We first introduce a notion which will be helpful for our upcoming proofs.

\subsection{Heavy Vertex Covers}

Recall the definition of essential subsets and superfluous vertices from Section~\ref{sec:homebasehypergraphs}.

The following is a particular type of vertex cover for home-base hypergraphs, which will be useful for the proofs in this and the next section.

\begin{defn}
Let $\cH$ be a home-base hypergraph on vertex classes $V_1$, $V_2$, and $V_3$ with a home-base partition $(\cF, \cR, W)$, and let $i, j \in \lset{1, 2, 3}$ with $i \neq j$. Let $C_i \subseteq W \cap V_i$ be the maximal essential set in $B_i$ and let $\cU_i \subseteq \cR$ be the set with $\abs{\cU_i} = \abs{C_i}$ and $N_{B_i}(\cU_i) = C_i$. Then the union of the sets 
\begin{itemize}
	\item $C_i \cup \left( \left( V(\cF) \cup V(\cR) \right) \cap V_i \right) $
	\item $\left( \bigcup_{R \in \cR \setminus \cU_i} R \right) \cap V_j$
\end{itemize}
is called the \emph{$i$-heavy $(i, j)$-cover} of $\cH$.
\end{defn}

\begin{obs} \label{obs:heavysuperfluous}
Every vertex in $V_i$ which is not in the $i$-heavy $(i, j)$-cover is a superfluous vertex in $W \cap V_i$.
\end{obs}

\begin{prop}
If $\cH$ is a home-base hypergraph on vertex classes $V_1$, $V_2$, and $V_3$ with a home-base partition $(\cF, \cR, W)$, then for every pair $i, j \in \lset{1, 2, 3}$ with $i \neq j$, the $i$-heavy $(i, j)$-cover is a minimal vertex cover of $\cH_0$.
\end{prop}

\begin{proof}
Let $T$ be the $i$-heavy $(i, j)$-cover of $\cH$. Let $e \in E(\cH)$. Then by the edge-home property, $e$ is at home in some $F \in \cF$ or some $R \in \cR$. If it is at home in $F$, then it contains some vertex in $F \cap V_i$, and so it intersects $T$. If it is at home in $R \in \cR \setminus \cU_i$, then it contains some vertex in $R \cap (V_i \cup V_j)$, and hence intersects $T$. The only remaining case is that $e$ is at home in some $R' \in \cU_i$. Let $V_i \cap e = \lset{v}$. If $v \in V(\cF) \cup V(\cR)$, then $e$ intersects $T$. If $v \in W \cap V_i$, then $vR'$ is an edge of $B_i$, and hence $v \in N_{B_i}(\cU_i) = C_i$, which shows that $e$ again intersects $T$. Thus $T$ is a vertex cover of $\cH$.

We now calculate the size of $T$. By the definition of the $i$-heavy $(i, j)$-cover, we get $\abs{T} = 2\abs{\cF} + \abs{\cR} + \abs{C_i} + \abs{\cR} - \abs{\cU_i}$. Since $\abs{C_i} = \abs{\cU_i}$, we get $\abs{T} = 2\abs{\cF} + 2\abs{\cR} = 2\abs{\cF \cup \cR} = 2\nu(\cH)$, and because home-base hypergraphs are tight for Ryser's Conjecture by Proposition~\ref{prop:homebasetight}, we get $\abs{T} = \tau(\cH)$ as desired.
\end{proof}

\subsection{Facts About Cromulent Triples}

We start with some lemmas about cromulent and perfectly cromulent triples. Note that properties~\ref{cromulent:neighborhood} and~(5*) make the roles of $Y_1$ and $Y_2$ symmetric in perfectly cromulent triples. This gives us the following observation:

\begin{obs} \label{obs:cromulentsymmetry}
$(Y_1, Y_2, X)$ is a perfectly cromulent triple if and only if $(Y_1, Y_2, X)$ and $(Y_2, Y_1, X)$ are both cromulent triples.
\end{obs}

Most of the proofs in this section work for cromulent triples, and can be strengthened for perfectly cromulent triples by using Observation~\ref{obs:cromulentsymmetry}.

\begin{assm}
For the rest of this section, let $\cH$ be a $3$-partite $3$-uniform hypergraph with vertex classes $V_1$, $V_2$, and $V_3$ such that $\tau(\cH) = 2\nu(\cH)$, and assume it has a cromulent triple $(Y_1, Y_2, X)$. We will assume without loss of generality that $Y_1 \subseteq V_1$, $Y_2 \subseteq V_2$, and $X \subseteq V_3$. We also fix a hypergraph matching $M \subseteq E(\cH|_{Y_1 \cup Y_2 \cup X})$ of size $\abs{Y_1}$. Let $\cH_0 = \cH - (Y_1 \cup Y_2 \cup X)$ be the corresponding home-base hypergraph, and fix a home-base partition $(\cF, \cR, W)$ of $\cH_0$.
\end{assm}

\begin{lem} \label{lem:existsynrw}
For every pair $(i, j) \in \lset{(1, 2), (1, 3), (2, 1)}$ we have that for every $y \in Y_i$ there is an edge $ywu$, where $w \in W \cap V_j$, and $u \in V(\cH_0) \setminus V(\cR)$. If $(Y_1, Y_2, X)$ is perfectly cromulent, then this holds also for $(i, j) = (2, 3)$.
\end{lem}

\begin{proof}
We will construct a vertex set $T$ of size $\tau(\cH) - 1$ which intersects all edges of $\cH$ except for the edges of the form in question. Since $T$ cannot be a vertex cover by virtue of its small size, some such edge must exist. Let $T$ be the union of the sets $Y_1 \cup Y_2 \setminus \lset{y}$, $(V(\cF) \cup V(\cR)) \cap V_j$, and $V(\cR) \cap V_k$, where $k \in \lset{1, 2, 3} \setminus \lset{i, j}$. Since we have taken two vertices from each $F \in \cF$ and two vertices from each $R \in \cR$, and $2\abs{Y_1} - 1$ additional vertices, we get $\abs{T} = 2\abs{\cF \cup \cR} + 2\abs{Y_1} - 1 = 2\nu(\cH_0) + 2\abs{Y_1} - 1 = 2\nu(\cH) - 1 = \tau(\cH) - 1$, hence $T$ is not a vertex cover of $\cH$.

It is clear that $T$ includes a cover of all edges of $\cH_0$, so any uncovered edge must contain $y$ or intersect $X$. It turns out that any edge $e$ intersecting $X$ is also covered by $T$. If $i = 1$, then $e$ is covered by $N_{\link{\cH}{V_1}}(X) = Y_2 \subseteq T$. If $i = 2$, then $j = 1$ and $e$ is covered by $N_{\link{\cH}{V_2}}(X) \subseteq Y_1 \cup (V(\cF) \cup V(\cR)) \cap V_1 \subseteq T$. Therefore, any edge not covered by $T$ must contain $y$ and two vertices of $\cH_0$. The $V_j$-vertex must be a $W$-vertex because $(V(\cF) \cup V(\cR)) \cap V_j \subseteq T$, and the $V_k$-vertex cannot be in $V(\cR)$ because $V(\cR) \cap V_k \subseteq T$.
\end{proof}

\begin{lem} \label{lem:existsyxs}
For every pair $(i, j) \in \lset{(1, 2), (1, 3), (2, 1)}$ we have that for every $y \in Y_i$ there is an edge $ysu$, where $s \in W \cap V_j$ is superfluous, and $u \in V(\cH_0)$. If $(Y_1, Y_2, X)$ is perfectly cromulent, then this holds also for $(i, j) = (2, 3)$.
\end{lem}

\begin{proof}
We will construct a vertex set $T$ of size $\tau(\cH) - 1$ which intersects all edges of $\cH$ except for the edges of the form in question. Since $T$ cannot be a vertex cover by virtue of its small size, some such edge must exist. Let $T$ be the union of $Y_1 \cup Y_2 \setminus \lset{y}$ and the $j$-heavy $(j, i)$-cover of $\cH_0$. Since we have taken $\tau(\cH_0)$ vertices from $\cH_0$ and $2\abs{Y_1} - 1$ additional vertices, we get $\abs{T} = 2\abs{\cF \cup \cR} + 2\abs{Y_1} - 1 = \tau(\cH) - 1$ (as calculated before). As in the proof of Lemma~\ref{lem:existsynrw}, the $V_i$-vertex of any uncovered edge must be $y$, and the other vertices are in $V(\cH_0)$. The $V_j$-vertex of an uncovered edge must be a superfluous vertex because besides $(V(\cF) \cup V(\cR)) \cap V_j$, the maximal essential subset $C_j \subseteq W \cap V_j$ of $B_j$ is also included in $T$ (and every $W$-vertex outside of the maximal essential subset is by definition superfluous).
\end{proof}

\begin{lem} \label{lem:yxstoynrs}
For $i = 1$ and $j = 3$ we have that for every $y \in Y_i$, if $yvs$ is an edge of $\cH$ with $v \in V(\cH_0)$ and $s \in V_j$ a superfluous vertex, then there is an edge $yv's$ with $v' \in V(\cH_0) \setminus V(\cR)$. If $(Y_1, Y_2, X)$ is perfectly cromulent, then this holds also for $(i, j) = (2, 3)$.
\end{lem}

\begin{proof}
We may assume $v \in V(\cR)$, otherwise we are done. Let $y' \in Y_2$ be the $V_2$-vertex of the edge of $\cM$ containing $y$.

By Lemma~\ref{lem:existsynrw} (with $(i, j) = (2, 1)$ for $y' \in Y_2$), there is an edge $wy'u$ with $w \in W \cap V_1$ and $u \in V(\cH_0) \setminus V(\cR)$. We claim $s = u$.

Suppose not. Then $yvs$ and $wy'u$ are disjoint edges. We can apply Case~\ref{monstercor:rnrnr} of Corollary~\ref{cor:monstercor} with $a = v$, $b = w$, $c = u$, and $S = \lset{s}$ to find a matching of size $\nu(\cH_0)$ in $\cH_0 - \lset{s, u, v, w}$. This matching together with the edges $yvs$, $wy'u$, and the rest of $\cM$ (besides the edge containing $y$ and $y'$) forms a matching of size $\nu(\cH_0) + 2 + \abs{Y_i} - 1 = \nu(\cH) + 1$, a contradiction. Hence $s = u$.

By Lemma~\ref{lem:existsyxs} (with $(i, j) = (1, 2)$ for $y \in Y_1$), there is an edge $yv'u'$ with $v'$ a superfluous vertex in $W \cap V_2$. If $u' \neq s$, then $yv'u'$ and $wy's$ are disjoint edges. We can apply Case~\ref{monstercor:nenrx} of Corollary~\ref{cor:monstercor} with $a = v'$, $b = w$, $c = u'$, and $S = \lset{s}$ to find a matching of size $\nu(\cH_0)$ in $\cH_0 - \lset{s, v', w, u'}$. This matching together with the edges $yv'u'$, $wy's$, and the rest of $\cM$ (besides the edge containing $y$ and $y'$) forms a matching of size $\nu(\cH_0) + 2 + \abs{Y_i} - 1 = \nu(\cH) + 1$, a contradiction.

Therefore $u' = s$, and thus $yv's$ is the edge we are looking for.
\end{proof}

The next lemma is a strengthening of Lemma~\ref{lem:yxstoynrs} in two ways: we can require more of our third vertex, and we can apply it to more combinations of $i$ and $j$.

\begin{lem} \label{lem:yxstoyss}
For $i = 1$ and for every $j \in \lset{2, 3}$ we have that for every $y \in Y_i$, if $yvs$ is an edge of $\cH$ with $v \in V(\cH_0)$ and $s \in V_j$ a superfluous vertex, then there is an edge $ys's$ with $s'$ also superfluous. If $(Y_1, Y_2, X)$ is perfectly cromulent, then this holds also for $i = 2$ and $j \in \lset{1, 3}$.
\end{lem}

\begin{proof}
Let $yvs$ be an edge with $v \in V(\cH_0)$ and $s \in V_j$ superfluous. Let $y' \in Y_2$ be the $V_2$-vertex of the edge of $\cM$ containing $y$. There are two cases.

\noindent \textbf{Case 1}. $i = 1$, $j = 3$.

By Lemma~\ref{lem:yxstoynrs} (with $(i, j) = (1, 3)$), we may assume $v \in V(\cH_0) \setminus V(\cR)$. By Lemma~\ref{lem:existsyxs} (with $(i, j) = (2, 1)$ for $y' \in Y_2$), there is an edge $s''y'u$ with $s'' \in V_i$ a superfluous vertex. If $s \neq u$, then $yvs$ and $s''y'u$ are disjoint edges, and we will reach a contradiction as in the previous lemma. We can apply Case~\ref{monstercor:nenrx} of Corollary~\ref{cor:monstercor} with $a = s''$, $b = v$, $c = u$, and $S = \lset{s}$ to find a matching of size $\nu(\cH_0)$ in $\cH_0 - \lset{s, s'', u, v}$. This matching together with the edges $yvs$, $s''y'u$, and the rest of $\cM$ (besides the edge containing $y$ and $y'$) forms a matching of size $\nu(\cH_0) + 2 + \abs{Y_i} - 1 = \nu(\cH) + 1$, a contradiction.

It follows that $s = u$. Lemma~\ref{lem:existsyxs} (with $(i, j) = (1, 2)$ for $y \in Y_1$) tells us that there is an edge $ys'u'$ with $s' \in V_2$ superfluous. It must be the case that $s = u'$ because otherwise $ys'u'$ and $s''y's$ are disjoint edges, and we would reach a similar contradiction. We can apply Case~\ref{monstercor:nenrx} of Corollary~\ref{cor:monstercor} with $a = s''$, $b = s'$, $c = u'$, and $S = \lset{s}$ to find a matching of size $\nu(\cH_0)$ in $\cH_0 - \lset{s, s', s'', u'}$. This matching together with the edges $ys'u'$, $s''y's$, and the rest of $\cM$ (besides the edge containing $y$ and $y'$) forms a matching of size $\nu(\cH_0) + 2 + \abs{Y_i} - 1 = \nu(\cH) + 1$, a contradiction.

Therefore there is an edge $ys's$, as required.

\noindent \textbf{Case 2}. $i = 1$, $j = 2$.

By Lemma~\ref{lem:existsyxs}) (with $(i, j) = (1, 3)$ for $y \in Y_1$) there is an edge $yr's'$ with $s' \in V_3$ superfluous, and then by Case 1, above, there is an edge $yrs'$ with $r \in V_2$ and $s' \in V_3$ both superfluous. By Lemma~\ref{lem:existsyxs} (with $(i, j) = (2, 1)$ for $y' \in Y_2$), there is an edge $qy'u$ with $q \in V_1$ a superfluous vertex and $u \in V(\cH_0)$. If $u \neq s'$, then we will again reach a contradiction. Suppose $yrs'$ and $qy'u$ are disjoint. We can apply Case~\ref{monstercor:nenrx} of Corollary~\ref{cor:monstercor} with $a = q$, $b = r$, $c = u$, and $S = \lset{s'}$ to find a matching of size $\nu(\cH_0)$ in $\cH_0 - \lset{q, r, s', u}$. This matching together with the edges $yrs'$, $qy'u$, and the rest of $\cM$ (besides the edge containing $y$ and $y'$) forms a matching of size $\nu(\cH_0) + 2 + \abs{Y_i} - 1 = \nu(\cH) + 1$, a contradiction.

Therefore $u = s'$. A similar contradiction is reached by $ysv$ and $qy's'$ if $v \neq s'$, so that cannot be the case either. Suppose $ysv$ and $qy's'$ are disjoint. We can apply Case~\ref{monstercor:nenrx} of Corollary~\ref{cor:monstercor} with $a = q$, $b = s$, $c = v$, and $S = \lset{s'}$ to find a matching of size $\nu(\cH_0)$ in $\cH_0 - \lset{q, s, s', v}$. This matching together with the edges $ysv$, $qy's'$, and the rest of $\cM$ (besides the edge containing $y$ and $y'$) forms a matching of size $\nu(\cH_0) + 2 + \abs{Y_i} - 1 = \nu(\cH) + 1$, a contradiction.

Therefore we have found our edge $yss'$.
\end{proof}

\begin{lem} \label{lem:prince}
Let $y \in Y_1$ and $y' \in Y_2$ be in an edge of $\cM$ together. Then there is a unique superfluous vertex $z_{y,y'} \in V_3$ such that
\begin{enumerate}
	\renewcommand{\theenumi}{(\roman{enumi})}
	\renewcommand{\labelenumi}{\theenumi}
	\item \label{prince:existence} There are edges $yvz_{y,y'}$ and $uy'z_{y,y'}$ for some vertices $u, v \in V(\cH_0)$,
	\item \label{prince:uniqueness} If $yv's'$ or $u'y's'$ is an edge with $s'$ superfluous, then $s' = z_{y,y'}$.
\end{enumerate}
\end{lem}

\begin{proof}
By Lemma~\ref{lem:existsyxs} (with $(i, j) = (1, 3)$ for $y \in Y_1$) there is an edge $yvs$ with $v \in V(\cH_0)$ and $s \in V_3$ superfluous. We claim that $s$ satisfies~\ref{prince:existence} and~\ref{prince:uniqueness}.

To see~\ref{prince:existence}, we only need to find $uy's$, since we have $yvs$. By Lemma~\ref{lem:yxstoyss} (with $(i, j) = (1, 2)$), we may assume $v$ is superfluous as well. By Lemma~\ref{lem:existsyxs} (with $(i, j) = (2, 1)$ for $y' \in Y_2$), we have an edge $s'y'u'$ with $s' \in W \cap V_1$ superfluous. Suppose $u' \neq s$. Then $yvs$ and $s'y'u'$ are disjoint edges. We can apply Case~\ref{monstercor:nenrx} of Corollary~\ref{cor:monstercor} with $a = v$, $b = s'$, $c = u'$, and $S = \lset{s}$ to find a matching of size $\nu(\cH_0)$ in $\cH_0 - \lset{s, s', u', v}$. This matching together with the edges $yvs$, $s'y'u'$, and the rest of $\cM$ (besides the edge containing $y$ and $y'$) forms a matching of size $\nu(\cH_0) + 2 + \abs{Y_i} - 1 = \nu(\cH) + 1$, a contradiction.

Therefore $u' = s$, and we have the desired edge $s'y's$.

We now show~\ref{prince:uniqueness}. Let $yv's'$ and $u'y's''$ be edges of $\cH$ with $s', s'' \in V_3$ both superfluous vertices. By Lemma~\ref{lem:yxstoyss} (with $(i, j) = (1, 2)$), we may assume $v'$ is superfluous as well. If $s' \neq s''$, then $yv's'$ and $u'y's''$ are disjoint edges. This leads to a contradiction as before. We can apply Case~\ref{monstercor:nenrx} of Corollary~\ref{cor:monstercor} with $a = v'$, $b = s'$, $c = u'$, and $S = \lset{s''}$ to find a matching of size $\nu(\cH_0)$ in $\cH_0 - \lset{s', s'', u', v'}$. This matching together with the edges $yv's'$, $u'y's''$, and the rest of $\cM$ (besides the edge containing $y$ and $y'$) forms a matching of size $\nu(\cH_0) + 2 + \abs{Y_i} - 1 = \nu(\cH) + 1$, a contradiction.

Therefore it must be the case that $s' = s''$, which in particular means that $s' = s'' = s$, since we could have substituted $yvs$ or $uy's$ for $yv's'$ or $u'y's''$, respectively.
\end{proof}

Our aim is to make each set $\lset{y, y', z_{y, y'}}$ into an $R$ for our home-base partition. We will first show that the $z_{y, y'}$'s are all distinct, and then we will make use of Lemma~\ref{prop:matchableedgehome} to show that combining the new $R$'s with the home-base partition of $\cH_0$ forms a home-base partition of $\cH$.

\begin{lem} \label{lem:v3matching}
For each $(y, y')$-pair, the associated $z_{y, y'}$ is distinct, and there is a matching saturating $\cR$ in the subgraph of $B_3$ induced by $\cR \cup (V_3 \cap W \setminus Z)$, where $Z$ is the set of all $z_{y, y'}$'s.
\end{lem}

\begin{proof}
Define the bipartite graph $K$ with parts $\cR \cup Y_1$ and $W \cap V_3$, where there is an edge between $R \in \cR$ and $w \in W \cap V_3$ precisely when there is an $R$-edge containing $w$, and there is an edge between $y \in Y_1$ and $w \in W \cap V_3$ precisely when $w = z_{y, y'}$, where $y'$ is the partner of $y$ in the pairing between $Y_1$ and $Y_2$. We claim that $K$ has a matching saturating $\cR \cup Y_1$.

We will apply Hall's theorem, so let $\cR_0 \subseteq \cR$ and $Y_0 \subseteq Y_1$. We construct a vertex cover $T$ of $\cH$. Let $C_3$ be the maximal essential set in the subgraph of $K$ induced by $\cR$ and $W \cap V_3$ (this is the graph $B_3$ associated with $\cH_0$), and let $\cU_3 \subseteq \cR$ be such that $N_K(\cU_3) = C_3$, which exists by the definition of essential. Let $T$ be the union of the sets $(Y_1 \cup Y_2) \setminus Y_0$, $N_K(\cR_0 \cup Y_0)$, $(V(\cR) \cup V(\cF)) \cap V_3$, $C_3$, and $\bigcup_{R \in \cR \setminus (\cU_3 \cup \cR_0)} (R \cap V_1)$. Note the similarities to the $3$-heavy $(3, 1)$-cover of $\cH_0$.

We must show that $T$ is indeed a vertex cover. Let $e \in E(\cH_0)$. Then $e$ is either an $\cF$-edge or an $\cR$-edge. If it is an $\cF$-edge, it is covered by $V(\cF) \cap V_3 \subseteq T$. If it is an $\cR$-edge, then it is covered by $(V(\cF) \cup V(\cR) \cap V_3 \subseteq T$, unless its $V_3$-vertex is in $W$, so assume that is the case. Let $e$ be an $R$-edge. If $R \in \cR_0$, then $e \cap V_3 \in N_K(R) \subseteq T$. If $R \in \cU_3$, then $e \cap V_3 \in C_3 \subseteq T$. If $R \in \cR \setminus (\cU_3 \cup \cR_0)$, then $e \cap V_1 = R \cap V_1 \subseteq T$. This shows that $T$ covers every edge of $\cH_0$. All edges incident to $X$ intersect $Y_2$, so any uncovered edge must be incident to $Y_0$ and two vertices of $\cH_0$. All such edges whose $V_3$-vertex is not superfluous intersect $T$, since $C_3 \cup (V(\cR) \cup V(\cF)) \cap V_3 \subseteq T$. Thus, the only edges we have to worry about are those incident to some $y \in Y_0$ and a superfluous vertex in $V_3$. Then by Lemma~\ref{lem:prince}, the $V_3$-vertices of those edges are the corresponding $z_{y, y'}$, and hence those edges intersect $N_K(Y_0) \subseteq T$. This shows that $T$ is a vertex cover.

We now calculate the size of $T$. By the definition of $T$, we calculate $\abs{T} = \abs{Y_1} + \abs{Y_2} - \abs{Y_0} + \abs{N_K(\cR_0 \cup Y_0)} + 2\abs{\cF} + \abs{\cR} + \abs{C_3} - \abs{C_3 \cap N_K(\cR_0)} + \abs{\cR} - \abs{\cU_3 \cup \cR_0}$. Because it is a vertex cover, we must have $\abs{T} \geq \tau(\cH)$. Since $\nu(\cH) = \nu(\cH_0) + \abs{Y_1}$ by the definition of cromulent triple, and since $\tau(\cH) = 2\nu(\cH)$, we have $\tau(\cH) = 2\nu(\cH_0) + 2\abs{Y_1} = 2\abs{\cF \cup \cR} + \abs{Y_1} + \abs{Y_2}$. Combining this with the fact that $\tau(\cH) \leq \abs{T}$ yields the inequality $\abs{Y_0} + \abs{\cU_3 \cup \cR_0} + \abs{C_3 \cap N_K(\cR_0)} \leq \abs{N_K(\cR_0 \cup Y_0)} + \abs{C_3}$. By the inclusion-exclusion principle we can rewrite this as $\abs{Y_0} + \abs{\cU_3} + \abs{\cR_0} - \abs{\cU_3 \cap \cR_0} + \abs{C_3 \cap N_K(\cR_0)} \leq \abs{N_K(\cR_0 \cup Y_0)} + \abs{C_3}$. Since $C_3 = N_K(\cU_3)$, we clearly have $C_3 \cap N_K(\cR_0) \supseteq N_K(\cU_3 \cap \cR_0)$. Since $B_3$ has a matching saturating $\cR$, by Hall's Theorem, we must have $\abs{\cU_3 \cap \cR_0} \leq \abs{N_K(\cU_3 \cap \cR_0)}$. Combining this with our previous inequality, we then get $\abs{Y_0} + \abs{\cU_3} + \abs{\cR_0} - \abs{\cU_3 \cap \cR_0} + \abs{\cU_3 \cap \cR_0} \leq \abs{N_K(\cR_0 \cup Y_0)} + \abs{C_3}$, which simplifies to $\abs{Y_0} + \abs{\cR_0} \leq \abs{N_K(\cR_0 \cup Y_0)}$, since $\abs{\cU_3} = \abs{C_3}$. This last inequality shows that we can apply Hall's Theorem to find a matching in $K$ saturating $\cR \cup Y_0$, which proves the lemma.
\end{proof}

\begin{lem} \label{lem:v1v2matching}
For $i = 2$, let $K_i$ be the bipartite graph with parts $\cR \cup Y_{3 - i}$ and $W \cap V_i$, where there is an edge between $R \in \cR$ and $w \in W \cap V_i$ precisely when there is an $R$-edge containing $w$, and there is an edge between $y \in Y_{3 - i}$ and $w \in W \cap V_i$ precisely when there is an edge $ywz_{y, y'}$, where $y'$ is the partner of $y$ in the pairing between $Y_1$ and $Y_2$. Then $K_i$ has a matching saturating $\cR \cup Y_{3 - i}$. If $(Y_1, Y_2, X)$ is perfectly cromulent, then this holds also for $i = 1$.
\end{lem}

\begin{proof}
We will apply Hall's theorem, so let $\cR_0 \subseteq \cR$ and $Y_0 \subseteq Y_{3 - i}$. We construct a vertex cover $T$ of $\cH$. Let $C_i$ be the maximal essential set in the subgraph of $K_i$ induced by $\cR$ and $W \cap V_i$ (this is the graph $B_i$ associated with $\cH_0$), and let $\cU_i \subseteq \cR$ be such that $N_{K_i}(\cU_i) = C_i$, which exists by the definition of essential. Let $T$ be the union of the sets $(Y_1 \cup Y_2) \setminus Y_0$, $N_{K_i}(\cR_0 \cup Y_0)$, $(V(\cR) \cup V(\cF)) \cap V_i$, $C_i$, and $\bigcup_{R \in \cR \setminus (\cU_i \cup \cR_0)} (R \cap V_3)$. Note the similarities to the $i$-heavy $(i, 3)$-cover of $\cH_0$.

We must show that $T$ is indeed a vertex cover. Let $e \in E(\cH_0)$. Then $e$ is either an $\cF$-edge or an $\cR$-edge. If it is an $\cF$-edge, it is covered by $V(\cF) \cap V_i \subseteq T$. If it is an $\cR$-edge, then it is covered by $(V(\cF) \cup V(\cR) \cap V_i \subseteq T$, unless its $V_i$-vertex is in $W$, so assume that is the case. Let $e$ be an $R$-edge. If $R \in \cR_0$, then $e \cap V_i \in N_K(R) \subseteq T$. If $R \in \cU_i$, then $e \cap V_i \in C_i \subseteq T$. If $R \in \cR \setminus (\cU_i \cup \cR_0)$, then $e \cap V_3 = R \cap V_3 \subseteq T$. This shows that $T$ covers every edge of $\cH_0$. All edges incident to $X$ intersect $Y_2$, which if $i = 2$ is part of $T$, and if $i = 1$, then $(Y_1, Y_2, X)$ is assumed to be perfectly cromulent, in which case all edges incident to $X$ are incident to $Y_1 \subseteq T$. Therefore, any uncovered edge must be incident to $Y_0$ and two vertices of $\cH_0$. All such edges whose $V_3$-vertex is not superfluous intersect $T$, since $C_i \cup (V(\cR) \cup V(\cF)) \cap V_i \subseteq T$. Thus, the only edges we have to worry about are those incident to some $y \in Y_0$ and a superfluous vertex $s \in V_i$. By Lemma~\ref{lem:yxstoyss} (with $(i, j) = (3 - i, i)$), there is an edge containing $y$ and $s$, whose $V_3$-vertex is also superfluous. By Lemma~\ref{lem:prince}, the $V_3$-vertices of those edges are the corresponding $z_{y, y'}$, and hence their $V_2$-vertices are in $N_{K_i}(Y_0) \subseteq T$ by the definition of $K_i$. This shows that $T$ is a vertex cover.

We now calculate the size of $T$. By the definition of $T$, we calculate $\abs{T} = \abs{Y_1} + \abs{Y_2} - \abs{Y_0} + \abs{N_{K_i}(\cR_0 \cup Y_0)} + 2\abs{\cF} + \abs{\cR} + \abs{C_i} - \abs{C_i \cap N_{K_i}(\cR_0)} + \abs{\cR} - \abs{\cU_i \cup \cR_0}$. Because it is a vertex cover, we must have $\abs{T} \geq \tau(\cH)$. Since $\nu(\cH) = \nu(\cH_0) + \abs{Y_1}$ by the definition of cromulent triple, and since $\tau(\cH) = 2\nu(\cH)$, we have $\tau(\cH) = 2\nu(\cH_0) + 2\abs{Y_1} = 2\abs{\cF \cup \cR} + \abs{Y_1} + \abs{Y_2}$. Combining this with the fact that $\tau(\cH) \leq \abs{T}$ yields the inequality $\abs{Y_0} + \abs{\cU_i \cup \cR_0} + \abs{C_i \cap N_{K_i}(\cR_0)} \leq \abs{N_{K_i}(\cR_0 \cup Y_0)} + \abs{C_i}$. By the inclusion-exclusion principle we can rewrite this as $\abs{Y_0} + \abs{\cU_i} + \abs{\cR_0} - \abs{\cU_i \cap \cR_0} + \abs{C_i \cap N_{K_i}(\cR_0)} \leq \abs{N_{K_i}(\cR_0 \cup Y_0)} + \abs{C_i}$. Since $C_i = N_{K_i}(\cU_i)$, we clearly have $C_i \cap N_{K_i}(\cR_0) \supseteq N_{K_i}(\cU_i \cap \cR_0)$. Since $B_i$ has a matching saturating $\cR$, by Hall's Theorem, we must have $\abs{\cU_i \cap \cR_0} \leq \abs{N_{K_i}(\cU_i \cap \cR_0)}$. Combining this with our previous inequality, we then get $\abs{Y_0} + \abs{\cU_i} + \abs{\cR_0} - \abs{\cU_i \cap \cR_0} + \abs{\cU_i \cap \cR_0} \leq \abs{N_{K_i}(\cR_0 \cup Y_0)} + \abs{C_i}$, which simplifies to $\abs{Y_0} + \abs{\cR_0} \leq \abs{N_{K_i}(\cR_0 \cup Y_0)}$, since $\abs{\cU_i} = \abs{C_i}$. This last inequality shows that we can apply Hall's Theorem to find a matching in $K_i$ saturating $\cR \cup Y_0$, which proves the lemma.
\end{proof}

\subsection{The Proof of Corollary~\ref{cor:cromulenttohbh}}

It suffices to prove Lemmas~\ref{lem:induction} and~\ref{lem:perfectlycromulent}.

\begin{proof}[Proof of Lemma~\ref{lem:induction}]
Let $(Y_1, Y_2, X)$ be a perfectly cromulent triple. We set $\cR' = \cR \cup \set{\lset{y, y', z_{y, y'}}}{y \in Y_1, y' \in Y_2 \mbox{ in an edge of } \cM \mbox{ together with } y}$, and $W' = W \cup X \setminus \set{z_{y, y'}}{y \in Y_1, y' \in Y_2 \mbox{ in an edge of } \cM \mbox{ together with } y}$, where $z_{y,y'}$ is the superfluous vertex in $V_3$ from Lemma~\ref{lem:prince}. By applying Lemma~\ref{lem:v3matching}, we find that $(\cF, \cR', W')$ is an FR-partition, since $\nu(\cH) = \nu(\cH_0) + \abs{Y_1} = \abs{\cF \cup \cR} + \abs{Y_1} = \abs{\cF \cup \cR'}$. Applying \ref{lem:v1v2matching} for $i = 1, 2$ we get that $(\cF, \cR', W')$ has a matching in $B_1'$ and $B_2'$. We can combine the partial matching in $B_3'$ that we get from Lemma~\ref{lem:v3matching} with the edges of $\cM$ going to $X$ to complete it. Thus $(\cF, \cR', W')$ is a matchable FR-partition. It is clearly also proper, because there are no edges with three vertices in $W'$ by virtue of the fact that no such edge is in $\cH_0$ and all edges going to $X$ have their other vertices in $Y_1$ and $Y_2$. Thus, by Proposition~\ref{prop:matchableedgehome}, we in fact have a home-base partition.
\end{proof}

\begin{proof}[Proof of Lemma~\ref{lem:perfectlycromulent}]
Let $(Y_1, Y_2, X)$ be a cromulent triple. We now mean to rule out the possibility that any edge incident to $X$ is also incident to an $\cF$- or $\cR$-vertex of $\cH_0$. Lemma~\ref{lem:v1v2matching} means that we can find a hypergraph matching $\cM'$ of size $\abs{Y_1}$ in $\cH$ consisting of edges of the form $yss'$ with $y \in Y_1$, and $s, s'$ superfluous vertices in $\cH_0$. Suppose there were an edge $uy'x$ for some $u \in (V(\cF) \cup V(\cR)) \cap V_1$, $y' \in Y_2$, and $x \in X$. By the matchability of $B_1$, we can choose a matching of WRR-edges for each $R \in \cR$, which avoids $u$, since $u \notin W$. We can also clearly find a matching of $\cF$-edges avoiding $u$. Combining these matchings with $\cM'$ yields a hypergraph matching of size $\nu(\cH)$ which is disjoint from $uy'x$. This is impossible, so such an edge cannot exist. Therefore $(Y_1, Y_2, X)$ is a perfectly cromulent triple.
\end{proof}

Therefore, we have shown that if we have a cromulent triple, we have a home-base hypergraph. The next section is devoted to finding cromulent triples under various assumptions.

\section{Searching for Cromulent Triples} \label{sec:searchingforcromulenttriples}

Let $\cH$ be a $3$-partite $3$-graph with vertex classes $V_1$, $V_2$, and $V_3$, and with $\tau(\cH) = 2\nu(\cH)$. We want to find a home-base partition of $\cH$. By Corollary~\ref{cor:cromulenttohbh}, we are done if $\cH$ has a cromulent triple. Therefore, our goal will be to find a cromulent triple inside our hypergraph. We will do this under a few assumptions, and we will later show that if all of these assumptions fail to hold, then we can prove $\cH$ is a home-base hypergraph even without cromulent triples.

Finding cromulent triples will entail finding a subgraph which is a home-base hypergraph. We do this by finding a subgraph which is tight for Ryser's Conjecture and has a smaller matching number than $\cH$, and then applying induction on Theorem~\ref{thm:characterization}. We would like to pinpoint exactly where in the proof we need to rely on induction. Therefore, we lay out the induction hypothesis here precisely.

\begin{ih}[IH($k$)]
If $\cH$ is a $3$-partite $3$-graph with $\nu(\cH) \leq k$ and $\tau(\cH) = 2\nu(\cH)$, then $\cH$ is a home-base hypergraph.
\end{ih}

The first assumption under which we will find a cromulent triple is if we have a good set (see Definition~\ref{def:good}).

\subsection{Good Subsets Lead to Cromulent Triples}

\begin{lem} \label{lem:goodtocromulent}
Suppose IH($k - 1$) holds. Let $\cH$ be a $3$-partite $3$-graph with vertex classes $V_1$, $V_2$, and $V_3$ such that $\tau(\cH) = 2\nu(\cH) = 2k$. If $X \subseteq V_3$ is a good set for $\link{\cH}{V_1}$, then the triple $(Y_1, Y_2, X)$ is perfectly cromulent, where $Y_1 = N_{\link{\cH}{V_2}}(X)$ and $Y_2 = N_{\link{\cH}{V_1}}(X)$.
\end{lem}

\begin{proof}
Let $X \subseteq V_3$ be a good set, and let $Y_2 = N_{\link{\cH}{V_1}}(X)$. Let $y \in Y_2$, and let $\cH_y = \cH - \set{vyz \in E(\cH)}{v \in V_1, z \in V_3 \setminus X}$. Since the deleted edges can be covered by one vertex ($y$), we clearly have $\tau(\cH_y) \geq \tau(\cH) - 1$, and of course $\nu(\cH_y) \leq \nu(\cH)$ as $\cH_y \subseteq \cH$. It is easy to see that $\link{\cH_y}{V_1} = \link{\cH}{V_1} - \set{yz \in E(\link{\cH}{V_1})}{z \in V_3 \setminus X}$. Therefore, because $X$ is good, we have $\conn(L(\link{\cH_y}{V_1})) \geq \conn(L(\link{\cH}{V_1})) + 1$. Recall that by Theorem~\ref{thm:connoflink}, we have $\conn(L(\link{\cH}{V_1})) = \nu(\cH) - 2$. Thus, we in fact have $\conn(L(\link{\cH_y}{V_1})) \geq \nu(\cH) - 1$. By Proposition~\ref{prop:linkconn}, there is a subset $S \subseteq V_1$ for which we have $\conn(L(\link{\cH_y}{S})) \leq \nu(\cH_y) - (\abs{V_1} - \abs{S}) - 2$ and $\abs{S} \geq \abs{V_1} - (2\nu(\cH_y) - \tau(\cH_y))$.
(Note that $|V_1| \geq 2$, so Proposition~\ref{prop:linkconn} part (ii) applies.)
 Plugging in the inequalities for $\tau$ and $\nu$, we get
\[
	\conn(L(\link{\cH_y}{S})) \leq \nu(\cH) - (\abs{V_1} - \abs{S}) - 2
\]
and
\[
	\abs{S} \geq \abs{V_1} - (2\nu(\cH) - \tau(\cH) + 1) = \abs{V_1} - 1
\]
since $\tau(\cH) = 2\nu(\cH)$.

We have seen that $V_1$ itself does not fulfil the first of these inequalities, so $S$ must be a proper subset of $V_1$, and thus by the second inequality, $S = V_1 \setminus \lset{a}$ for some $a \in V_1$. A priori, we do not know if this $a$ is unique for each $y \in Y_2$, so denote by $A_y$ the set of all $V_1$-vertices $a$ for which $\conn(L(\link{\cH_y}{V_1 \setminus \lset{a}})) \leq \nu(\cH) - 3$.

Let $a \in A_y$ and let $S = V_1 \setminus \lset{a}$. By Theorem~\ref{thm:matchconn}, we have $\nu(\link{\cH_y}{S}) \leq 2\conn(L(\link{\cH_y}{S})) + 4 \leq 2\nu(\cH) - 2 = \tau(\cH) - 2$, which implies that $\nu(\link{\cH}{S}) \leq \tau(\cH) - 1$ because at most one edge of each maximum matching has been erased when passing from $\cH$ to $\cH_y$ in the link of $S$. We must have $\tau(\cH_y) = \tau(\cH) - 1$ because if $\tau(\cH_y) = \tau(\cH)$, then by inequality~\ref{linkconn:lower} of Proposition~\ref{prop:linkconn}, we would have $\conn(L(\link{\cH_y}{S})) \geq \tau(\cH_y)/2 - 2$ (since $\conn(L(\link{\cH_y}{S}))$ is an integer and $\tau(\cH_y) = \tau(\cH)$ is even), which is a contradiction. We can in fact show $\nu(\link{\cH}{S}) = \tau(\cH) - 1$, from which $\nu(\link{\cH_y}{S}) = \tau(\cH) - 2$ then follows, by considering the vertex cover $T_S$ of $\cH$ consisting of $a$ and a minimum vertex cover of $\link{\cH}{S}$ (which, by K\"onig's Theorem, has size $\nu(\link{\cH}{S})$).

This means that every maximum matching in $\link{\cH}{S}$ must contain an edge which is not in $\link{\cH_y}{S}$. Set $Z = V_3 \setminus X$ and $W = V_2 \setminus Y_2$. We get the following structure for the maximum matchings:

\begin{claim}
For every $y \in Y_2$ and for every $a \in A_y$ every maximum matching in $\link{\cH}{V_1 \setminus \lset{a}}$ contains an edge $yz$ for some $z \in Z$, and then saturates $Y_2 \setminus \lset{y}$ using $(X, Y_2)$-edges and saturates $Z \setminus \lset{z}$ using $(Z, W)$-edges.
\end{claim}

\begin{proof}
Let $S = V_1 \setminus \lset{a}$. As observed, every maximum matching in $\link{\cH}{S}$ contains an edge from $y$ to $Z$. Since $X$ is good (hence decent), it satisfies property~\ref{decent:nu} of Definition~\ref{def:decent}, so $\nu(\link{\cH}{V_1}) = \abs{Y_2} + \abs{Z}$. Then because there are no edges between $X$ and $W$, it follows that every maximum matching in $\link{\cH}{V_1}$ saturates $Y_2$ with edges incident to $X$ and saturates $Z$ with edges incident to $W$. Since $\nu(\link{\cH}{S}) = \tau(\cH) - 1 = \nu(\link{\cH}{V_1}) - 1$, we cannot have more than one matching edge between $Y_2$ and $Z$. Thus the claim follows.
\end{proof}

This structure immediately implies that the sets $A_y$ are pairwise disjoint.

\begin{claim}
If $y, y' \in Y_2$ with $y \neq y'$, then $A_y \cap A_{y'} = \emptyset$.
\end{claim}

\begin{proof}
Let $a \in A_y$, and let $S = V_1 \setminus \lset{a}$. Then we know that a maximum matching in $\link{\cH}{S}$ contains a $(y, Z)$-edge and the rest of its edges are between $X$ and $Y_2$ and between $Z$ and $W$. Thus the only edge between $Y_2$ and $Z$ in the matching is incident to $y$. For $a' \in A_{y'}$, the structure of the maximum matchings in $\link{\cH}{V_1 \setminus \lset{a'}}$ is different, and thus $a \neq a'$, hence the sets $A_y$ and $A_y'$ must be disjoint.
\end{proof}

Since every $A_y$ is non-empty, we thus clearly have $\abs{\bigcup_{y \in Y_2} A_y} \geq \abs{Y_2}$.

\begin{claim}
For every $a \in \bigcup_{y \in Y_2} A_y$, every maximum $(X, Y_2)$-matching in $\link{\cH}{V_1}$ must have one edge which extends only to $a$.
\end{claim}

\begin{proof}
Suppose there were a maximum $(X, Y_2)$-matching $M'$ in $\link{\cH}{V_1}$ in which every edge extended to an element of $S = V_1 \setminus \lset{a}$. Then we could take a maximum $(V_2, V_3)$-matching in $\link{\cH}{S}$ (which must contain a $(y, Z)$-edge) and replace the part of the matching which hits $Y_2$ with $M'$. Because $X$ has no neighbors outside of $Y_2$, this modified matching is a matching and is at least as big as the original one and therefore also maximum. This does not use a $(y, Z)$-edge, so we have a contradiction. Thus $M'$ must contain an edge which does not extend to $S$, and hence extends only to $a$.
\end{proof}

From this claim, we see that $\abs{\bigcup_{y \in Y_2} A_y} = \abs{Y_2}$, since there can be at most as many vertices in $\bigcup_{y \in Y_2} A_y$ as edges in a maximum $(X, Y_2)$-matching in $\link{\cH}{V_1}$, of which there are precisely $\abs{Y_2}$.

\begin{claim}
$Y_1 = \bigcup_{y \in Y_2} A_y$ and there is a hypergraph matching in $\cH_{Y_1 \cup Y_2 \cup X}$ saturating $Y_1$ and $Y_2$.
\end{claim}

\begin{proof}
We clearly have $Y_1 \supseteq \bigcup_{y \in Y_2} A_y$ by the previous claim. We will show the other inclusion as well. Consider any vertex $x \in Y_1$. It follows from the definitions of $Y_1$ and $Y_2$ that there is an $(X, Y_2)$-edge $e$ in $\link{\cH}{V_1}$ such that $e \cup \lset{x} \in E(\cH)$. Since $X$ is good, $e$ appears in a maximum matching $M$. For every $y \in Y_2$ and every $a \in A_y$, one edge of the matching between $X$ and $Y_2$ must extend to $a$ (recall that to be maximum, $M$ must saturate $Y_2$ using $(Y_2, X)$-edges and must saturate $Z$ using $(Z, W)$-edges). Since the $A_y$'s are all disjoint, the matching extends to a hypergraph matching saturating $Y_2$ and $\bigcup_{y \in Y_2} A_y$. Since $e$ extends to $\bigcup_{y \in Y_2} A_y$, it follows that $x \in \bigcup_{y \in Y_2} A_y$ and hence $Y_1 = \bigcup_{y \in Y_2} A_y$. This proves the claim.
\end{proof}

Now we almost have that $(Y_1, Y_2, X)$ is perfectly cromulent. We just need to show that $\cH_0 = \cH \setminus (Y_1 \cup Y_2 \cup X)$ is a home-base hypergraph with $\nu(\cH_0) = \nu(\cH) - \abs{Y_1}$.

Consider the graph $\cH_1 = \cH \setminus (Y_1 \cup Y_2)$. Since we have removed only $2\abs{Y_1}$ vertices from $\cH$, it follows that $\tau(\cH_1) \geq \tau(\cH) - 2\abs{Y_1}$. We must have $\nu(\cH_1) \leq \nu(\cH) - \abs{Y_1}$ because to any matching in $\cH_1$, we may add the matching of size $\abs{Y_1}$ we just showed exists to it to produce a matching in $\cH$ (because no matching edge in the original matching is incident to $Y_1 \cup Y_2 \cup X$). Because $\tau(\cH_1) \leq 2\nu(\cH_1)$, we must have equality in both cases, whence $\tau(\cH_1) = 2\nu(\cH_1) = 2\nu(\cH) - 2\abs{Y_1}$. Note however that $X$ is a set of isolated vertices in $\cH_1$, and so removing them changes neither the matching size nor the covering number. Hence $\cH_0 = \cH_1 \setminus X$ also has $\tau(\cH_0) = 2\nu(\cH_0) = 2\nu(\cH) - 2\abs{Y_1}$. By induction on the matching number of the Ryser-tight hypergraph, $\cH_0$ is a home-base hypergraph. This proves that $(Y_1, Y_2, X)$ is a perfectly cromulent triple.
\end{proof}

This lemma shows that if $\link{\cH}{V_i}$ has a good set for any $i$, then we find a perfectly cromulent triple.

\subsection{No Good Sets}

From now on we assume that $\link{\cH}{V_1}$ has no good set. Recall that by 
Theorem~\ref{thm:connoflink}, we know that $\conn(L(\link{\cH}{V_1})) = \nu(\cH) - 2$, and so by Lemma~\ref{lem:goodsets} $\link{\cH}{V_1}$ has a perfect  matching. Moreover for every minimal equineighbored set $X \subseteq V_3$ both it and its neighborhood $N_{\link{\cH}{V_1}}(X)$ have size $2$ and together induce a $C_4$ (possibly with parallel edges). Our next assumption will be that there are two disjoint hyperedges incident to some minimal equineighbored set.

\begin{lem} \label{lem:twodisjointhyperedges}
Suppose IH($k - 1$) holds. Let $\cH$ be a $3$-partite $3$-graph with vertex classes $V_1$, $V_2$, and $V_3$ such that $\tau(\cH) = 2\nu(\cH) = 2k$, and let $\link{\cH}{V_1}$ have no good sets. Suppose there is a minimal equineighbored set $X \subseteq V_3$ in $\link{\cH}{V_1}$ such that there are two disjoint hyperedges $zyx$ and $z'y'x'$ of $\cH$ with $x, x' \in X$. Let $Y_1 = \lset{z, z'} \subseteq V_1$ and $Y_2 = \lset{y, y'} \subseteq V_2$. Then $(Y_1, Y_2, X)$ is a cromulent triple.
\end{lem}

\begin{proof}
For Condition~\ref{cromulent:size} note that $\abs{Y_1} = \abs{Y_2} = \abs{X} = 2$, since by Lemma~\ref{lem:goodsets} $X$ has size $2$.

Then $X = \lset{x, x'}$ and because $X$ is equineighbored, the neighborhood of $X$ is also of size $2$, that is,  
$N_{\link{\cH}{V_1}}(X) = \lset{y, y'}$. So Condition~\ref{cromulent:neighborhood} is satisfied.

For Condition~\ref{cromulent:matching} note that by assumption there are two disjoint hyperedges $zyx$ and $z'y'x'$ in $\cH|_{Y_1 \cup Y_2 \cup X}$ and that $\abs{Y_1} = 2$.

For Condition~\ref{cromulent:homebase} we first prove that $\tau(\cH_0) = 2\nu(\cH_0) = 2(\nu(\cH) - \abs{Y_1})$. Then we can use IH($k - 1$) to derive the existence of a home-base partition of $\cH_0$. First, consider the graph $\cH_1 = \cH \setminus (Y_1 \cup Y_2)$. Since we have removed only $2\abs{Y_1}$ vertices from $\cH$, it follows that $\tau(\cH_1) \geq \tau(\cH) - 2\abs{Y_1}$. We must have $\nu(\cH_1) \leq \nu(\cH) - \abs{Y_1}$ because $X$ consists of isolated vertices in $\cH_1$, so we may add $zyx$ and $z'y'x'$ to any matching in $\cH_1$ to obtain a matching $2$ larger in $\cH$. Because $\tau(\cH_1) \leq 2\nu(\cH_1)$, we must have equality in both cases, whence $\tau(\cH_1) = 2\nu(\cH_1) = 2\nu(\cH) - 2\abs{Y_1}$. Note however that because $X$ is a set of isolated vertices in $\cH_1$, removing them changes neither the matching size nor the covering number. Hence $\cH_0 = \cH_1 \setminus X$ also has $\tau(\cH_0) = 2\nu(\cH_0) = 2\nu(\cH) - 2\abs{Y_1}$. Thus $\cH_0$ has a home-base partition $(\cF, \cR, W)$.

The proof of Condition~\ref{cromulent:edges} is far more involved and will use a number of internal lemmas, so we give a brief overview. Our goal will be to find a contradiction by providing a larger matching than $\nu(\cH)$ if there is an edge of $\cH$ incident to $X$ and a $W$-vertex of $\cH_0$. This matching will consist of a maximum matching in $\cH_0$ and a few extra edges whose existence will be guaranteed by the high vertex cover number of $\cH$. We utilize the fact that we are quite flexible in choosing a matching for $\cH_0$, so that we can usually avoid the vertices of the extra edges when we choose our matching. Recall the definition of superfluous vertices and $i$-heavy $(i, j)$-covers from Section~\ref{sec:cromulenttriples}.

\begin{lem} \label{lem:xynotw}
There is no edge $wyx$ with $w\in W$. Similarly there is no $wy'x'$.
\end{lem}

\begin{proof}
Suppose $wyx$ is an edge. Take the following partial cover of $\cH$: $y$, $y'$, and $z'$ plus the $2$-heavy $(2, 3)$-cover of $\cH_0$. Since this set of vertices is one too small to be a cover, this implies the existence of an edge $zsp$ avoiding it, where $s$ is superfluous in $\cH_0$, and $p \in V(\cH_0)$. Indeed, an edge not intersecting the partial cover must avoid $Y_2$, hence also $X$, is not in $E(\cH_0)$, and by Observation~\ref{obs:heavysuperfluous}, its $V_2$-vertex is superfluous. By Case~\ref{monstercor:nenrx} of Corollary~\ref{cor:monstercor} applied to $\cH_0$ with $a = s$, $b = w$, $c = p$, and $S = \emptyset$, we can find a matching of size $\nu(\cH_0)$ inside $\cH_0$ avoiding $\lset{s, w, p}$. This matching together with the edges $z'y'x'$, $wyx$, and $zsp$ gives a matching of size $\nu(\cH_0) + 3 = \nu(\cH) + 1$, a contradiction.
\end{proof}

\begin{lem} \label{lem:twoparticularedges}
If there is an edge of $\cH$ incident to $X$ and a vertex of $W \cap V_1$, then there are two disjoint edges of $\cH$ whose $V_1$-vertices are in $W$, at least one being superfluous, whose $V_2$-vertices are $y$ and $y'$, and exactly one of whose $V_3$-vertices are in $V(\cH_0)$.
\end{lem}

\begin{proof}
Suppose there is an edge incident to $w \in W \cap V_1$ and $X$. Without loss of generality suppose it is incident to $x$. Then by Lemma~\ref{lem:xynotw}, it is not incident to $y$, so it must be the edge $wy'x$.

Suppose that $w$ is superfluous in $\cH_0$. Then we will show that $wyx'$ is also an edge of $\cH$ and that $wy'x$ and $wyx'$ are the only edges extending $y'x$ or $yx'$.

Since $X$ is a minimal equineighbored of size $2$, we have $yx' \in E(\link{\cH}{V_1})$, and hence there is some edge $vyx' \in E(\cH)$. Suppose $v \neq w$. Take the partial cover consisting of $\lset{y, y'}$ plus the $2$-heavy $(2, 3)$-cover of $\cH_0$. If $v \in \lset{z, z'}$, then add $v$ to the partial cover. If $v \in R_1 \in \cR$, then add instead the vertex in $R_1 \cap V_3$ to the partial cover. This leaves an edge of the form $(z \mbox{ or } z')sp$ where $s \in V_2$ is superfluous in $\cH_0$ and $p \notin R_1$ (in case $v \in V(\cR)$, hence $R_1$ exists) which is disjoint from $vyx'$. Indeed, an edge not intersecting the partial cover must avoid $Y_2$, hence also $X$, is not in $E(\cH_0)$, and by Observation~\ref{obs:heavysuperfluous}, its $V_2$-vertex is superfluous. If $v \in \lset{z, z'}$, then we can apply Case~\ref{monstercor:nenrx} of Corollary~\ref{cor:monstercor} to $\cH_0$ with $a = w$, $b = s$, $c = p$, and $S = \emptyset$. If $v \in V(\cR)$, then we can apply Case~\ref{monstercor:rnrnr} of Corollary~\ref{cor:monstercor} to $\cH_0$ with $a = v$, $b = p$, $c = s$, and $S = \lset{w}$. And if $v \in V(\cH_0) \setminus V(\cR)$, then we can apply Case~\ref{monstercor:nenrx} of Corollary~\ref{cor:monstercor} to $\cH_0$ with $a = s$, $b = v$, $c = p$, and $S = \lset{w}$. In any case, we find a matching in $\cH_0$ of size $\nu(\cH_0)$ avoiding $\lset{w, v, s, p}$. Then this matching together with $wy'x$, $vyx'$, and $(z \mbox{ or } z')sp$ gives a matching of size $\nu(\cH_0) + 3 = \nu(\cH) + 1$, a contradiction.

Therefore the only edge extending $yx'$ is $wyx'$, and because $wyx'$ is an edge, a similar argument shows that $wy'x$ is the only edge extending $y'x$.

Take a partial cover $\lset{z, z', w}$ plus the $1$-heavy $(1,2)$-cover of $\cH_0$. This leaves an edge $w'(y \mbox{ or } y')p$ where $w'$ is superfluous and $w' \neq w$. Indeed, an edge not intersecting the partial cover is not in $E(\cH_0)$, and by Observation~\ref{obs:heavysuperfluous}, its $V_1$-vertex is superfluous. Also $p \notin \lset{x, x'}$, since $w' \neq w$. It is disjoint from one of $wyx'$ and $wy'x$, so $w'(y \mbox{ or } y')p$ together with whichever of $wyx'$ and $wy'x$ it is disjoint from are the two disjoint edges we are after.

Suppose on the other hand, that there is no edge incident to $\lset{x, x'}$ which extends to a superfluous vertex in $V_1$. Then in particular $w$ is not superfluous in $\cH_0$. Take the partial cover $\lset{z, z', y'}$ plus the $1$-heavy $(1, 3)$-cover of $\cH_0$. This leaves an edge $syp$ where $s$ is superfluous in $\cH_0$, and hence $s \neq w$. Indeed, an edge not intersecting the partial cover is not in $E(\cH_0)$, and by Observation~\ref{obs:heavysuperfluous}, its $V_1$-vertex is superfluous. Also $p \notin \lset{x, x'}$, since $s$ is superfluous. Thus $wy'x$ and $syp$ are the two disjoint edges we are after.
\end{proof}

Thus, suppose that there is an edge incident to $W \cap V_1$ and $X$. By Lemma~\ref{lem:twoparticularedges}, there are two disjoint edges $e$ and $f$ whose vertices intersect $V(\cH_0)$ in $s, w \in W \cap V_1$ and $p \in V_3$. At least one of $s$ and $w$ is superfluous in $\cH_0$, so suppose without loss of generality that $s$ is the superfluous one. We consider several cases, depending on the location of $p$. In each case we will reach a contradiction.

\noindent \textbf{Case 1}. $p \in V(\cF)$.

Take the partial cover $\lset{y, y', z}$, plus the $3$-heavy $(3, 2)$-cover of $\cH_0$. This gives an edge $z'p's'$ where $s'$ is superfluous (hence $s' \neq p$). Indeed, an edge not intersecting the partial cover must avoid $Y_2$, hence also $X$, is not in $E(\cH_0)$, and by Observation~\ref{obs:heavysuperfluous}, its $V_3$-vertex is superfluous. We can apply Case~\ref{monstercor:fwx} of Corollary~\ref{cor:monstercor} with $a = p$, $b = w$, $c = p'$, and $S = \lset{s, s'}$ to obtain a matching of size $\nu(\cH_0)$ in $\cH_0$ avoiding $\lset{s, s', w, p', p}$. This matching together with the edges $e$, $f$, and $z'p's'$ gives a matching of size $\nu(\cH_0) + 3 = \nu(\cH) + 1$, a contradiction.

\noindent \textbf{Case 2}. $p \in R_1 \in \cR$.

Take the partial cover $\lset{y, y'}$ together with the vertex in $R_1 \cap V_2$ and the $3$-heavy $(3, 2)$-cover of $\cH_0$. This gives an edge $(z \mbox{ or } z')p's'$ where $s'$ is superfluous (note $s' \neq p$) and $p'$ is not in $R_1$. Indeed, an edge not intersecting the partial cover must avoid $Y_2$, hence also $X$, is not in $E(\cH_0)$, and by Observation~\ref{obs:heavysuperfluous}, its $V_3$-vertex is superfluous. We can apply Case~\ref{monstercor:rnrnr} of Corollary~\ref{cor:monstercor} with $a = p$, $b = p'$, $c = w$, and $S = \lset{s, s'}$ to obtain a matching of size $\nu(\cH_0)$ in $\cH_0$ avoiding $\lset{s, s', w, p', p}$. This matching together with the edges $e$, $f$, and $(z \mbox{ or } z')p's'$ gives a matching of size $\nu(\cH_0) + 3 = \nu(\cH) + 1$, a contradiction.

\noindent \textbf{Case 3}. $p \in W$ is essential for $R_1 \in \cR$.

Take the partial cover $\lset{y, y'}$, the $V_2$-vertex essential for $R_1$ if it exists, plus the $3$-heavy $(3, 2)$-cover of $\cH_0$. This gives an edge $(z \mbox{ or } z')p's'$ where $s'$ is superfluous (hence $s' \neq p$) and $p'$ is not essential for $R_1$. Indeed, an edge not intersecting the partial cover must avoid $Y_2$, hence also $X$, is not in $E(\cH_0)$, and by Observation~\ref{obs:heavysuperfluous}, its $V_3$-vertex is superfluous. We can apply Case~\ref{monstercor:enenr} of Corollary~\ref{cor:monstercor} with $a = p$, $b = p'$, $c = w$, and $S = \lset{s, s'}$ to obtain a matching of size $\nu(\cH_0)$ in $\cH_0$ avoiding $\lset{s, s', w, p', p}$. This matching together with the edges $e$, $f$, and $(z \mbox{ or } z')p's'$ gives a matching of size $\nu(\cH_0) + 3 = \nu(\cH) + 1$, a contradiction.

\noindent \textbf{Case 4}. $p \in W$ is not essential but not superfluous.

Take the partial cover $\lset{y, y'}$ plus the $3$-heavy $(3, 2)$-cover of $\cH_0$. This gives an edge $(z \mbox{ or } z')p's'$ where $s'$ is superfluous, hence $s' \neq p$. Indeed, an edge not intersecting the partial cover must avoid $Y_2$, hence also $X$, is not in $E(\cH_0)$, and by Observation~\ref{obs:heavysuperfluous}, its $V_3$-vertex is superfluous. By Lemma~\ref{lem:superfluousessential}, $p$ does not become essential after removing a superfluous vertex from $V_3$. Then we can apply Case~\ref{monstercor:nenrx} of Corollary~\ref{cor:monstercor} with $a = p$, $b = w$, $c = p'$, and $S = \lset{s, s'}$ to obtain a matching of size $\nu(\cH_0)$ in $\cH_0$ avoiding $\lset{s, s', w, p', p}$. This matching together with the edges $e$, $f$, and $(z \mbox{ or } z')p's'$ gives a matching of size $\nu(\cH_0) + 3 = \nu(\cH) + 1$, a contradiction.

\noindent \textbf{Case 5}. $p \in W$ is superfluous.

Take the partial cover $\lset{y, y', p}$ plus the $2$-heavy $(2, 3)$-cover of $\cH_0$. This gives an edge $(z \mbox{ or } z')s'p'$ where $s'$ is superfluous and $p' \neq p$. Indeed, an edge not intersecting the partial cover must avoid $Y_2$, hence also $X$, is not in $E(\cH_0)$, and by Observation~\ref{obs:heavysuperfluous}, its $V_2$-vertex is superfluous. We can apply Case~\ref{monstercor:nenrx} of Corollary~\ref{cor:monstercor} with $a = s'$, $b = w$, $c = p'$, and $S = \lset{s, p}$ to obtain a matching of size $\nu(\cH_0)$ in $\cH_0$ avoiding $\lset{s, s', w, p', p}$. This matching together with the edges $e$, $f$, and $(z \mbox{ or } z')s'p'$ gives a matching of size $\nu(\cH_0) + 3 = \nu(\cH) + 1$, a contradiction.

Thus we conclude that there can be no edge incident to $W \cap V_1$ and $X$, so Condition~\ref{cromulent:edges} must hold, and hence $(Y_1, Y_2, X)$ is a cromulent triple.
\end{proof}

Thus, if we either have a good set, or if we have no good set and there are two disjoint hyperedges incident to a minimal equineighbored subset of some link graph, then we find a cromulent triple, and hence have found a home-base partition by Corollary~\ref{cor:cromulenttohbh}. Therefore, the only hypergraphs left to check are those which have no good set and where the hyperedges incident to any equineighbored subset of any link graph form intersecting hypergraphs. This case is handled in the next section.

\section{The End Game} \label{sec:theendgame}

We start with the following easy proposition which will be useful in what is to come:

\begin{prop} \label{prop:extendtof}
Let $\cH$ be a $3$-partite $3$-graph with vertex classes $V_1$, $V_2$, and $V_3$ such that each link $\link{\cH}{V_i}$ has a perfect matching. Suppose $X \subseteq V_j$ is a minimal equineighbored set of $\link{\cH}{V_i}$ with $\abs{X} = 2$, and suppose $X$ is not incident to two disjoint edges of $\cH$. Then the edges incident to $X$ form a truncated multi-Fano plane.
\end{prop}

\begin{proof}
Since $X$ is a minimal equineighbored set of size $2$ and $\link{\cH}{V_i}$ has no isolated vertices, it follows easily that the edges of $\link{\cH}{V_i}$ incident to $X$ form a $C_4$ (possibly with parallel edges). By assumption, the edges incident to $X$ form an intersecting hypergraph. Since the hyperedges incident to $X$ all intersect, each pair of opposite edges in the $C_4$ must extend to one vertex in $V_i$. If this is the same vertex $v$ for all pairs, then $N_{\link{\cH}{V_k}}(X) = \lset{v}$, where $V_k$ is the third vertex class besides $V_i$ and $V_j$. This contradicts the fact that $\link{\cH}{V_k}$ has a perfect matching, so each pair extends to a different vertex, which gives the truncated Fano plane. If there are parallel edges in the $C_4$, this analysis shows that they also have to extend to the same vertex as the edges to which they are parallel, hence we have a truncated multi-Fano plane.
\end{proof}

We aim to prove the following lemma, which is the missing ingredient in our proof of Theorem~\ref{thm:characterization}.

\begin{lem} \label{lem:endgame}
Suppose IH($k - 1$) holds. Let $\cH$ be a $3$-partite $3$-graph with vertex classes $V_1$, $V_2$, and $V_3$ such that $\tau(\cH) = 2\nu(\cH) = 2k$. Suppose that $\cH$ does not have a cromulent triple. Then there is an $X \subseteq V_3$, which is a minimal equineighbored set for $\link{\cH}{V_1}$ such that for its neighborhood $Y = N_{\link{\cH}{V_1}}(X)$ we also have $N_{\link{\cH}{V_1}}(Y) = X$.
\end{lem}

\begin{proof}
We have shown in Lemma~\ref{lem:goodtocromulent} that we have a cromulent triple if there is at least one good set, which means we are working under the assumption that $\link{\cH}{V_1}$ has no good set. By Lemma~\ref{lem:goodsets}, we then know that $\link{\cH}{V_1}$ has a perfect matching and that every minimal equineighbored set is of size $2$ and hence is incident to a $C_4$. Therefore, it is clear that every edge incident to a minimal equineighbored set participates in a perfect matching, so we have shown that every minimal equineighbored set is still decent.

If $X \subseteq V_3$ is a minimal equineighbored set, for $y \in N_{\link{\cH}{V_1}}(X)$ define the bipartite graph $G_y = \link{\cH}{V_1} - \set{yz \in E(\link{\cH}{V_1})}{z \in V_3 \setminus X}$. Since $X$ is decent but not good, it must be that for some $y \in N_{\link{\cH}{V_1}}(X)$ we have
\[
	\conn(L(G_y) \leq \conn(L(\link{\cH}{V_1})).
\]
A similar statement holds if $X \subseteq V_2$.

Now suppose for the sake of contradiction to the statement of Lemma~\ref{lem:endgame} that for every minimal equineighbored subset $X$ in $\link{\cH}{V_1}$, its neighborhood $Y$ has neighbors outside of $X$. Again, Theorem~\ref{thm:connoflink} gives that $\link{\cH}{V_1}$ is extremal for Theorem~\ref{thm:matchconn}, and hence it has a CP-decomposition by Theorem~\ref{thm:cpdecomposition}. We know that any CP-decomposition of $\link{\cH}{V_1}$ contains some $P_4$'s, since otherwise the graph would consist entirely of disjoint $C_4$'s, which is not the case if there are edges between $Y$ and $V_3 \setminus X$.

\begin{claim}
The graph $\link{\cH}{V_1}$ contains a minimal equineighbored set $X \subseteq V_3$ for which both elements of $N(X)$ have neighbors outside $X$ in $\link{\cH}{V_1}$.
\end{claim}

\begin{proof}
Let $Z$ be the set of endpoints of $P_4$'s in $V_3$ for some CP-decomposition of $\link{\cH}{V_1}$ with respect to some perfect matching $M$. Then $Z$ is equineighbored because the edges incident to the endpoints in $V_3$ all must contain an interior vertex in $V_2$ either of the same $P_4$ or of some other one. The set of interior vertices of $P_4$'s in $V_2$ is matched by $M$ to the set of endpoints of $P_4$'s in $V_3$, so these are the same size. Therefore $\abs{Z} = \abs{N(Z)}$. Since $Z$ is equineighbored, it contains a minimal equineighbored subset $X$.

Since $X$ consists of endpoints of $P_4$'s and $N(X)$ consists of interior vertices of $P_4$'s, the vertices in $N(X)$ all have neighbors outside $X$: the other interior vertices of their respective $P_4$'s.
\end{proof}

Fix a perfect matching $M$ of the link graph $\link{\cH}{V_1}$. Let $X_3 \subseteq V_3$ be a minimal equineighbored set for which both elements of $N(X_3)$ have neighbors outside $X_3$, and let $N(X_3) = \lset{y, y'}$. Let $X_3 = \lset{x, x'}$ so that $yx, y'x' \in M$. Without loss of generality, let $y'$ be a vertex of $N(X_3)$ that witnesses the failure of $X_3$ to be good; that is, we have
\[
	\conn(L(G_{y'})) \leq \conn(L(\link{\cH}{V_1})).
\]
Then by Theorem~\ref{thm:cpdecomposition}, $G_{y'}$ has a CP-decomposition with respect to $M$ (since no edges of $M$ were erased, and hence $G_{y'}$ is still extremal for Theorem~\ref{thm:matchconn}). We claim that in every CP-decomposition of $G_{y'}$, the two vertices of $X_3$ are together in one of the $C_4$'s of the decomposition. The edge $x'y'$ is an edge of $M$, so it must be in some $C_4$ or $P_4$ of the CP-decomposition. Since $N_{G_{y'}}(y') = X$, and $N_{G_{y'}}(x') = N_{\link{\cH}{V_1}}(X_3)$, this $C_4$ or $P_4$ must be contained in $G_{y'}[X_3 \cup N(X_3)]$. But we know the edges in $G_{y'}[X_3 \cup N(X_3)]$ form a $C_4$, so $x'y'$ can't be contained in a $P_4$ of the CP-decomposition (one of the edges $xy'$ and $x'y$ would not be at home anywhere).

Let $Z_2$ be the set of vertices in $V_2$ reachable by $M$-alternating paths in $G_{y'}$ starting at $y$ with an edge not in $M$ (including $y$ itself). Note that $Y \subseteq Z_2$. We have $\abs{N_{G_{y'}}(Z_2)} = \abs{Z_2}$ because every vertex of $V_3$ we reach is matched to a vertex of $V_2$ which is included in $Z_2$. Then $Z_2$ contains a minimal equineighbored set $X_2$. Note that $X_2$ is disjoint from $Y$, since $X_2 \setminus Y$ must also be equineighbored (because $X_3$ is taken out of the neighborhood), and $X_2 \setminus Y$ is not empty because $\abs{N_{G_{y'}}(Y)} > 2$. This means also that $X_2$ has exactly the same neighborhood in $G_{y'}$ and in $\link{\cH}{V_1}$, and so it is also a minimal equineighbored set for $\link{\cH}{V_1}$. Therefore, $\abs{X_2} = 2$ and the edges incident to $X_2$ form a $C_4$.

\begin{lem} \label{endsofps}
In any CP-decomposition of $G_{y'}$ all vertices of $Z_2 \setminus N(X_3)$ are endpoints of $P_4$'s, and all vertices of $N(Z_2 \setminus N(X_3))$ are interior vertices of $P_4$'s.
\end{lem}

\begin{proof}
Since the $(y', V_3 \setminus X_3)$-edges are erased, any CP-decomposition of $G_{y'}$ must have a $C_4$ on $X_3 \cup N(X_3)$. So any $M$-alternating path going out from $y$ (not to $X_3$) must go first to an interior vertex of a $P_4$, which is matched to an endpoint of that $P_4$, and so on, alternating between interior vertices and endpoints. So the neighbors of $Z_2 \setminus N(X_3)$ are interior vertices and the vertices of $Z_2 \setminus N(X_3)$ are endpoints.
\end{proof}

This shows in particular that both vertices of $X_2$ are endpoints of $P_4$'s, and both vertices of $N(X_2)$ are interior vertices of $P_4$'s, and hence both have neighbors outside of $X_2$.

\begin{lem} \label{extendtosame}
If $X \subseteq V_3$ and $X' \subseteq V_2$ are minimal equineighbored subsets of $\link{\cH}{V_1}$ with $X' \cap N(X) = \emptyset$, and there is an $M$-alternating path from $N(X)$ to $N(X')$ starting with a non-matching edge, then the edges incident to $X$ and the edges incident to $X'$ extend to the same two vertices $\lset{z, z'} \subseteq V_1$.
\end{lem}

\begin{proof}
We have seen that each link graph $\link{\cH}{V_i}$ has a perfect matching, and we know $\abs{X} = 2$ and is not incident to two disjoint hyperedges, so by Proposition~\ref{prop:extendtof}, the edges incident to $X$ form a truncated Fano plane.

Let $N(X) = \lset{y, y'}$, and let $N(X') = \lset{w, w'}$, where without loss of generality $y$ is the last vertex of $N(X)$ visited on the $M$-alternating path, and $w$ is the first vertex of $N(X')$ visited. Let $G_{y', w'}$ be the graph formed by erasing \emph{both} the $(y', V_3 \setminus X)$-edges and the $(w', V_2 \setminus X_2)$-edges from $\link{\cH}{V_1}$. We will show that $G_{y', w'}$ does \emph{not} have a CP-decomposition. Suppose it did. Then fix a CP-decomposition of $G_{y', w'}$. Both $X$ and $X'$ would need to consist of vertices of a $C_4$ in the CP-decomposition of $G_{y', w'}$, as previously observed for $G_{y'}$. However since there is an $M$-alternating path from $y$ to $w$ starting with a non-matching edge, we will see that this leads to a contradiction. Consider the first edge $yv$ of this path. It is not an edge of a $C_4$ or $P_4$ of the CP-decomposition, so it must be at home in some $P_4$, and since $y$ is not an interior vertex of a $P_4$ of the CP-decomposition, it follows that $v$ is. The next edge is an edge of $M$ which pairs the interior vertex $v$ with an endpoint. The next edge must be at home in some $P_4$, hence its other vertex is again an interior vertex of that $P_4$. Continuing in this manner, one sees that the even vertices of the path ($y$ being the first vertex) are interior vertices of $P_4$'s of the CP-decomposition. However, since $w$ is one of the even vertices, this contradicts the fact that $w$ is a vertex of a $C_4$ of the CP-decomposition. Therefore no CP-decomposition is possible, and hence by the contrapositive of Theorem~\ref{thm:cpdecomposition}, we must have
\begin{equation} \label{gy'w'}
	\conn(L(G_{y', w'})) \geq \frac{\nu(G_{y', w'})}{2} - 1 = \frac{\nu(\link{\cH}{V_1})}{2} - 1 = \nu(\cH) - 1,
\end{equation}
where the last equality is by Theorem~\ref{thm:connoflink}.

Consider the hypergraph $\cH_{y', w'}$ that results by removing from $\cH$ the edges inducing the $(y', V_3 \setminus X)$-edges and the $(w', V_2 \setminus X_2)$-edges in $\link{\cH}{V_1}$. Then clearly $\link{\cH_{y', w'}}{V_1} = G_{y', w'}$. We have $\tau(\cH_{y', w'}) \geq \tau(\cH) - 2$, since we can cover all of the deleted edges with two vertices, and we clearly have $\nu(\cH_{y', w'}) \leq \nu(\cH)$. Therefore by parts~\ref{linkconn:upper} and~\ref{linkconn:size} of Proposition~\ref{prop:linkconn}, there is some $S \subseteq V_1$ such that $\conn(L(\link{\cH_{y', w'}}{S})) \leq \nu(\cH) - (\abs{V_1} - \abs{S}) - 2$ and $\abs{S} \geq \abs{V_1} - 2$. 
(Note that if $|V_1| >2$ then Proposition~\ref{prop:linkconn} is applicable, and otherwise
the conclusion of the lemma is immediate.)
We know $S \neq V_1$ because the first inequality fails for $V_1$, as we have just concluded in the preceding paragraph.

Combining the inequality for $\conn(L(\link{\cH_{y', w'}}{S}))$ with the inequality in 
Theorem~\ref{thm:matchconn} gives that $\nu(\link{\cH_{y', w'}}{S}) \leq 2\nu(\cH) - 2(\abs{V_1} - \abs{S})$. Recalling the vertex cover $T_S$ of $\cH$ consisting of $V_1 \setminus S$ and a minimal vertex cover of $\link{\cH}{S}$ gives that $\nu(\link{\cH}{S}) \geq \tau(\cH) - (\abs{V_1} - \abs{S})$ (by K\"onig's Theorem). Thus we have
\begin{equation} \label{smallernu}
	\nu(\link{\cH_{y', w'}}{S}) \leq \nu(\link{\cH}{S}) - (\abs{V_1} - \abs{S}).
\end{equation}

Therefore, every maximum matching of $\link{\cH}{S}$ has to contain an edge that gets erased in $\cH_{y', w'}$. If $xy$ and $x'y'$ are in $\link{\cH}{S}$, then we can change any maximum matching to avoid a $(y', V_3 \setminus X)$-edge without changing the cardinality of the matching, and similarly for $xy'$ and $x'y$. Analogously, we can avoid a $(w', V_2 \setminus X')$-edge if either pair of opposite edges of the $C_4$ incident to $X'$ is contained in $\link{\cH}{S}$. Therefore for one of the $C_4$'s, no pair of opposite edges is contained in $\link{\cH}{S}$. This implies that the two vertices of $V_1$ to which the edges of the $C_4$ extend are not in $S$, and hence in fact $\abs{S} = \abs{V_1} - 2$.

This of course means that every maximum matching of $\link{\cH}{S}$ has to contain \emph{two} edges that get erased in $\cH_{y', w'}$, so no pair of opposite edges of either $C_4$ is contained in $\link{\cH}{S}$, and hence the vertices of $V_1$ to which the edges extend are not in $S$. But each $C_4$ extends to exactly two vertices, as observed in Lemma~\ref{prop:extendtof}, and since $\abs{S} = \abs{V_1} - 2$, they must be the same two vertices for $X$ and $X'$, as claimed.
\end{proof}

Lemma~\ref{extendtosame} applied to $X_2$ and $X_3$ shows that $\cH$ has two truncated Fano planes intersecting in two vertices $\lset{z, z'} \subseteq V_1$. We will see that this leads to a contradiction.

Let $X_2 = \lset{v, v'}$, and let $N(X_2) = \lset{w, w'}$. Assume without loss of generality that the truncated Fano planes consist of the edges $\lset{zyx, zy'x', z'yx', z'y'x}$ and $\lset{zvw, zv'w', z'vw', z'v'w}$. Consider the hypergraph $\cH' = \cH - \lset{y, w, z, z'}$, and note that $X_3$ and $X_2$ consist of isolated vertices in $\cH'$, since all edges incident to them are incident to $\lset{z, z'}$. Because we have deleted only four vertices, we clearly have $\tau(\cH') \geq \tau(\cH) - 4$. To any matching in $\cH'$ we may add $zyx$ and $z'vw$ to get a matching two larger in $\cH$, so we must have $\nu(\cH') \leq \nu(\cH) - 2$. Combining this with the assumption that $\tau(\cH) = 2\nu(\cH)$ and the fact that Ryser's Conjecture is true for $3$-partite hypergraphs we get the following sequence of inequalities:
\[
	\tau(\cH') \leq 2\nu(\cH') \leq 2\nu(\cH) - 4 = \tau(\cH) - 4 \leq \tau(\cH').
\]
Since the first and last expressions are the same, all inequalities are actually equalities, and hence $\cH'$ is also extremal for Ryser's Conjecture, with $\nu(\cH') = k - 2$. Therefore, by the inductive hypothesis IH($k - 1$), $\cH'$ has a home-base partition $(\cF, \cR, W)$.

We will find either a vertex cover of size $\tau(\cH) - 1$, or a matching of size $\nu(\cH) + 1$ in $\cH$, either of which gives our desired contradiction.

Consider the minimal vertex cover of $\cH'$ consisting of $V(\cF) \cap V_1$ and $V(\cR) \cap (V_1 \cup V_3)$. If adding the three vertices $z$, $z'$, and $w$ to this set would form a vertex cover $T$ of $\cH$, we would have a contradiction and be done, so we may assume that there is some edge $e \in E(\cH)$ which avoids $T$. Its $V_1$-vertex must be in $W$, since $(V(\cF) \cup V(\cR)) \cap V_1 \cup \lset{z, z'} \subseteq T$. Its $V_3$-vertex must be in $V(\cF) \cup W$, since $V(\cR) \cap V_3 \cup \lset{w} \subseteq T$ and any edge incident to $X_3$ intersects $T$ in $\lset{z, z'}$. Its $V_2$-vertex cannot be in $V(\cH')$, since otherwise $e$ would be an edge of $\cH'$ and hence intersect $T$, and its $V_2$-vertex also cannot be in $X_2$, since all edges incident to $X_2$ intersect $T$ in $\lset{z, z'}$. Therefore $e$ must go through $y$, so it is of the form $ayb$ for some vertices $a \in W \cap V_1$ and $b \in (V(\cF) \cup W) \cap V_3$.

Suppose we can find a maximum matching in $\cH'$ avoiding $a$, $y'$, and $b$. Then this matching plus the three disjoint edges $zy'x'$, $z'v'w$, and $ayb$ would form a matching of size $\nu(\cH) + 1$ in $\cH$, a contradiction.

By the monster lemma (Lemma~\ref{lem:monster}), we can find a matching of size $\nu(\cH')$ in $\cH' - \lset{a, y', b}$ if there is an $F$-edge avoiding $\lset{a, y', b}$ for each $F \in \cF$, and an $R$-edge avoiding $\lset{a, y', b}$ for each $R \in \cR$. Since $a \in W$, and $y'$ and $b$ are in different vertex classes, we do not cover all $F$-edges for any $F \in \cF$. Since $a, b \notin V(\cR)$, we could pick an RWR-edge for any $R \in \cR$ avoiding $\lset{a, y', b}$ \emph{unless} $y'$ is a $W$-vertex essential for some $R \in \cR$. This means that if $y' \notin W$, we have the desired contradictory matching, and hence we may assume $y' \in W$.

Consider the $1$-heavy $(1, 3)$-cover of $\cH'$ (see Section~\ref{sec:cromulenttriples} for the definition), which is a minimal vertex cover of $\cH'$. If adding the three vertices $z$, $z'$, and $w$ to this set would form a vertex cover $T'$ of $\cH$, we would again have a contradiction, so we may assume that some edge $e' \in E(\cH)$ avoids $T'$. Its $V_1$-vertex must be a superfluous $W$-vertex, since all other $V_1$-vertices are in $T'$. Its $V_3$-vertex must be in $V(\cH')$, since $w \in T'$ and any edge incident to $X_3$ intersects $T'$ in $\lset{z, z'}$. Its $V_2$-vertex cannot be in $V(\cH')$, since otherwise $e'$ would be an edge of $\cH'$ and hence intersect $T'$, and its $V_2$-vertex also cannot be in $X_2$, since all edges incident to $X_2$ intersect $T'$ in $\lset{z, z'}$. Therefore $e'$ must go through $y$, so it is of the form $a'yb'$ for some superfluous vertex $a' \in W \cap V_1$ and some vertex $b' \in V(\cH') \cap V_3$.

By part~\ref{monstercor:nenrx} of Corollary~\ref{cor:monstercor} of the monster lemma applied to $\cH'$ with $a = a'$, $b = y'$, and $c = b'$, there is a matching of size $\nu(\cH')$ in $\cH'$ avoiding $a'$, $y'$, and $b'$. Combining this matching with the three disjoint edges $zy'x'$, $z'v'w$, and $a'yb'$ yields a matching of size $\nu(\cH) + 1$, a contradiction.

Therefore, in all cases we have found a contradiction, and since we have assumed the negation of the statement of Lemma~\ref{lem:endgame}, we have proven the lemma.
\end{proof}

\section{The Proof of Theorem~\ref{thm:characterization}} \label{sec:theproofoftheoremcharacterization}

\begin{proof}[Proof of Theorem~\ref{thm:characterization}]
The proof is by induction. IH($0$) holds trivially: Let $\cH$ be a $3$-partite $3$-graph with $\nu(\cH) = 0$. Then $\cH$ has no edges, so $(\emptyset, \emptyset, V(\cH))$ is a home-base partition of $\cH$ as can easily be seen. Now assume IH($k - 1$) holds. We will show IH($k$).

Let $\cH$ be a $3$-partite $3$-graph with vertex classes $V_1$, $V_2$, and $V_3$ such that $\tau(\cH) = 2\nu(\cH) = 2k$. If it has a cromulent triple, then by Corollary~\ref{cor:cromulenttohbh}, it is a home-base hypergraph, and we are done.

Therefore, assume there is no cromulent triple. Then by Lemma~\ref{lem:endgame} there is a minimal equineighbored $X \subseteq V_3$ such that for $Y = N_{\link{\cH}{V_1}}(X)$ we also have $N_{\link{\cH}{V_1}}(Y) = X$. By Proposition~\ref{prop:extendtof}, the edges incident to $X$ form a truncated Fano plane $F$. Let $A$ be the set of $V_1$-vertices of the hyperedges of $F$. Set $\cH_1 = \cH \setminus A$. Since we have removed two vertices, we have $\tau(\cH_1) \geq \tau(\cH) - 2$, and since any matching in $\cH_1$ can be enlarged by adding an edge of $F$ (as no edge of $\cH_1$ is incident to $X$ or $Y$), we have $\nu(\cH_1) \leq \nu(\cH) - 1$. Combining these inequalities with the fact that $\tau(\cH_1) \leq 2\nu(\cH_1)$ yields that all three inequalities are actually equalities. Since $X$ and $Y$ consist of isolated vertices, the same holds true for $\cH_0 = \cH_1 \setminus (Y \cup X)$. Thus, we can apply induction to get a home-base partition of $\cH_0$ and add the $F$ to it to get a proper matchable FR-partition of $\cH$, which by Lemma~\ref{prop:matchableedgehome} is a home-base partition.

Thus in all cases, $\cH$ is a home-base hypergraph, so IH($k$) holds.

Therefore Theorem~\ref{thm:characterization} holds by induction.
\end{proof}

For interest, we can directly show also that IH($1$) holds.

\begin{prop} \label{prop:basecase}
Let $\cH$ be a $3$-partite $3$-graph with $\nu(\cH) = 1$ and $\tau(\cH) = 2$. Then $\cH$ is a home-base hypergraph.
\end{prop}

\begin{proof}
Suppose $\cH$ is an intersecting $3$-partite $3$-graph with $\tau(\cH) = 2$. If every pair of edges intersect in two vertices, then it is easy to see that there must then be two vertices which are in every edge, and thus $\cH$ would in fact have a vertex cover of size $1$ (pick any one of the two vertices). Therefore there must be two edges which intersect in one vertex. Label these edges $abc$ and $ade$. Since $a$ alone does not form a vertex cover, there must be an edge which misses $a$, but it must intersect both of these edges, each in a different vertex class of $\cH$. Thus WLOG, we have the edge $fbe$. If $fdc$ is also an edge of $\cH$, then we have an $F$. In this case, no further edge can be present unless it is parallel to one of the existing edges, since no other edge can intersect all four of these edges. Therefore in this case, $\cH$ is indeed a home-base hypergraph which consists of a single $F$.

If $fdc$ is not an edge of $\cH$, then we let $R = \lset{a, b, e}$, and we claim that every edge of $\cH$ contains at least two of the vertices $a$, $b$, or $e$. If an edge misses any two of these vertices, then its third vertex must be the vertex outside of $R$ of the edge among $abc$, $ade$, and $fbe$ that contains those two vertices (since $\cH$ is intersecting). Since this vertex is not in $R$ either, by symmetry the same is true of each of the other edges we have given. Thus the edge must in fact be $fdc$, which is not the case by assumption. Thus $(\emptyset, \lset{R}, V(\cH) \setminus R)$ forms an FR-partition of $\cH$ with the edge-home property. It is matchable because the graphs $B_1$, $B_2$, and $B_3$ contain edges $Rf$, $Rd$, and $Rc$, respectively, which obviously form matchings saturating $\lset{R}$. Thus in this case, $\cH$ is a home-base hypergraph consisting of a single $R$ and at least three $W$-vertices. This proves the case $\nu(\cH) = 1$.
\end{proof}

\section{Concluding Remarks and Open Questions} \label{sec:concludingremarks}

\subsection{Proof of the Reverse Implication for Theorem~\ref{thm:cpdecomposition}}

As promised, we prove here the ``if'' direction of Theorem~\ref{thm:cpdecomposition}.

\begin{proof}[Proof of Theorem~\ref{thm:cpdecomposition}, ($\Leftarrow$)]
Let $G$ be a bipartite graph with a collection of $\nu(G)/2$ pairwise vertex-disjoint subgraphs, each of them a $C_4$ or a $P_4$, such that every edge of $G$ is either an edge of one of the $C_4$'s or is incident to an interior vertex of one of the $P_4$'s. We will construct a home-base hypergraph $\cH$ with $G$ as one of its links.

Let $V_1$ and $V_2$ be the vertex classes of $G$. Let $V_3$ be a set of sufficiently many new vertices ($\nu(G)$ suffice). Let $\cH$ be the empty $3$-graph. Then $(\cF, \cR, W) = (\emptyset, \emptyset, \emptyset)$ is a home-base partition of $\cH$. We will add edges to $\cH$, maintaining a home-base partition $(\cF, \cR, W)$.

For each $C_4$ in the collection we do the following. Let $\lset{a, b, c, d}$ be the vertices of the $C_4$, so that $a, c \in V_1$, $b, d \in V_2$, and $ab, bc, cd, da \in E(G)$. Take two unused vertices $e, f \in V_3 \setminus V(\cH)$, and add the edges $abe$, $adf$, $cbf$, and $cde$ to $\cH$. These edges form a truncated Fano plane. For each edge parallel to an edge of the $C_4$, add an edge parallel to the corresponding one of these edges to $\cH$, forming a truncated multi-Fano plane. We can then add the set $F = \lset{a, b, c, d, e, f}$ to $\cF$, maintaining that $(\cF, \cR, W)$ is a home-base partition of $\cH$. Clearly, the $C_4$ is now present in the link $\link{\cH}{V_3}$ together with all its parallel edges.

Then, for each $P_4$ in the collection we do the following. Let $\lset{a, b, c, d}$ be the vertices of the $P_4$, so that $a, c \in V_1$, $b, d \in V_2$, and $ab, bc, cd \in E(G)$. Take two unused vertices $e, f \in V_3 \setminus V(\cH)$, and add the edges $abe$, $cbf$, and $cde$ to $\cH$. For each edge parallel to an edge of the $P_4$, add an edge parallel to the corresponding one of these edges to $\cH$. Add the set $R = \lset{b, c, e}$ to $\cR$, and add the vertices $a$, $d$, and $f$ to $W$. The edges $abe$, $cbf$, and $cde$ are $R$-edges with a $W$-vertex in $V_1$, $V_3$, and $V_2$, respectively. Thus $a$, $d$, and $f$ can be matched to $R$ in $B_1$, $B_3$, and $B_2$, respectively, without disturbing matchability, since the $W$-vertices are new. Clearly the $P_4$ is now present in the link $\link{\cH}{V_3}$ along with all parallel edges, and note especially that its interior vertices are members of $R$.

Once we've processed all the $C_4$'s and $P_4$'s, any edges of $G$ not yet present in the link $\link{\cH}{V_3}$ are incident to an interior vertex of one of the $P_4$'s. Let $xy \in E(G)$ be such an edge, and suppose $y$ is an interior vertex of one of the $P_4$'s. Then $y \in R$ for some $R \in \cR$. Let $z \in R \cap V_3$. Then, we add the edge $xyz$ to $\cH$. If $x$ was not previously a vertex of $\cH$, we add it to $W$, otherwise, we leave it where it is. Since $xyz$ is an $R$-edge, $\cH$ is still a home-base hypergraph with home-base partition $(\cF, \cR, W)$. After this addition, $xy$ is present in the link $\link{\cH}{V_3}$. We process every remaining edge this way.

If $G$ has any isolated vertices, we add them to $\cH$, putting them in $W$ (these clearly do not disturb the home-base partition of $\cH$). Now $\cH$ is a home-base hypergraph with $\link{\cH}{V_3} = G$. We know $\cH$ satisfies $\tau(\cH) = 2\nu(\cH)$ by Proposition~\ref{prop:homebasetight}, and hence by equation~\eqref{eq:linkconn} we have $\conn(L(G)) = \frac{\nu(G)}{2} - 2$, as desired.
\end{proof}

\subsection{The Connectedness of the Line Graphs of Home-Base Hypergraphs}

For $3$-graphs $\cH$, Theorem~\ref{thm:matchconn} gives
\[
	\conn(L(\cH)) \geq \frac{\nu(\cH)}{3} - 2.
\]
Using our characterization, we can show that the Ryser-extremal $3$-graphs are far from tight for this theorem. For a Ryser-extremal $3$-partite $3$-graph we can improve the bound to the following:

\begin{prop} \label{prop:homebaseconn}
If $\cH$ is a home-base hypergraph, then
\[
	\conn(L(\cH)) \geq \frac{2}{3}\nu(\cH) - 2.
\]
\end{prop}

It is not difficult to show that this bound is tight. The proof of Proposition~\ref{prop:homebaseconn} can be found in~\cite{lotharsthesis}. It makes use of Theorem~\ref{thm:characterization} as well as some topological notions discussed in \cite{HNS}, and hence is outside the scope of the present paper.

Since Proposition~\ref{prop:homebaseconn} is a strengthening of Theorem~\ref{thm:matchconn} when $\tau(\cH) = 2\nu(\cH)$, one could ask for the best possible extension of it when the ratio $\tau/\nu$ is different from $2$. To make this precise, let us define the function $f: [1, 2] \to \R$ by
\[
	f(x) = \inf\set{\frac{\conn(L(\cH)) + 2}{\nu(\cH)}}{\cH \mbox{ is a } 3\mbox{-partite } 3\mbox{-graph}, \tau(\cH) \geq x\nu(\cH)}.
\]
We then have that for any $3$-partite $3$-graph $\cH$ with $\tau(\cH) = x\nu(\cH)$ it holds that
\[
	\conn(L(\cH)) \geq f(x)\nu(\cH) - 2.
\]

Clearly $f$ is monotone increasing and bounded below by $1/3$, by Theorem~\ref{thm:matchconn}. Since Proposition~\ref{prop:homebaseconn} is tight, we have $f(2) = 2/3$, while there are easy examples showing $f(1) = 1/3$. One could speculate whether there is a linear lower bound on $f$ interpolating these two extremes, so that $f(x) \geq x/3$. This would be very interesting, as it would imply Ryser's Conjecture for $4$-partite $4$-graphs by a straightforward generalization of Aharoni's argument for $3$-partite $3$-graphs. Unfortunately, this does not turn out to be the case, as there is a violation of this bound for $x = 4/3$:

\begin{prop}
There is a $3$-partite $3$-graph $\cH$ with $\tau(\cH) = 4$ and $\nu(\cH) = 3$ such that $\conn(L(\cH)) = -1$.
\end{prop}

\begin{proof}
Let $\cH$ be the $3$-partite $3$-graph on the vertices $\lset{1, 2, 3} \times \lset{1, 2, 3, 4}$ with vertex classes given by the first coordinate. The edges are two intersecting matchings of size $3$ and a matching of size $2$ which interesects every edge of the first two matchings. The first matching is $\lset{(1, 1), (2, 2), (3, 3)}$, $\lset{(1, 2), (2, 3), (3, 1)}$, and $\lset{(1, 3), (2, 1), (3, 2)}$; the second is $\lset{(1, 2), (2, 4), (3, 3)}$, $\lset{(1, 3), (2, 2), (3, 4)}$, and $\lset{(1, 4), (2, 3), (3, 2)}$; and the two remaining edges are $\lset{(1, 2), (2, 2), (3, 2)}$, and $\lset{(1, 3), (2, 3), (3, 3)}$. It is not hard to check that $\tau(\cH) = 4$ and $\nu(\cH) = 3$. Since the last two edges intersect all the other ones, they form a connected component of size $2$ in $I(L(\cH))$, so the complex is not $0$-connected.

\begin{center}
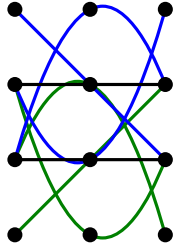

\begin{tikzpicture}
[edge1/.style={very thick, green!50!black},
edge2/.style={very thick, blue},
edge3/.style={very thick},
vertex/.style={circle, fill, inner sep=0, minimum size=2mm}]

\draw[edge1] (0,2) parabola bend (7/6,-1/24) (2,1);
\draw[edge1] (0,1) parabola bend (5/6,49/24) (2,0);
\draw[edge1] (0,0) -- (2,2);
\draw[edge2] (0,3) -- (2,1);
\draw[edge2] (0,2) parabola bend (5/6,23/24) (2,3);
\draw[edge2] (0,1) parabola bend (7/6,73/24) (2,2);
\draw[edge3] (0,1) -- (2,1);
\draw[edge3] (0,2) -- (2,2);
\foreach \x in {0, 1, 2}
	\foreach \y in {0, 1, 2, 3}
		\node[vertex] at (\x,\y) {};
\end{tikzpicture}
\captionof{figure}{A $3$-partite $3$-graph $\cH$ with $\tau(\cH) = 4$, $\nu(\cH) = 3$, and $\conn(L(\cH)) = -1$.}
\end{center}
\end{proof}

This shows that $f(x) = 1/3$ for $x \in [1, 4/3]$. It can also be shown that $f(x) \geq x/5$ for every $x \in [1, 2]$, but this only represents an improvement when $x \in (\frac{5}{3}, 2)$ (see~\cite{lotharsthesis}). We conjecture that $f(x) \geq x/4$ for every $x\in [1, 2]$. 

To approach Ryser's Conjecture for $4$-graphs, we seem to need a much better understanding of the potential link $3$-graphs, in particular those with $\tau(\cH) > \nu(\cH)$. We believe the function $f$ will be a useful tool for this purpose, even though the extension of Aharoni's argument, at least in its most straightforward version, does not succeed due to the fact that $f(4/3) = 1/3$.

\end{document}